\documentclass[11pt, letterpaper]{article}
\usepackage[letterpaper,margin=1.0in]{geometry}

\usepackage[utf8]{inputenc}

\usepackage{amsmath,amsfonts,bm}
\usepackage{amssymb}

\def\eqref#1{equation~(\ref{#1})}

\def\floor#1{\left\lfloor #1 \right\rfloor}
\def\1{\bm{1}}

\newcommand{\Norm}[1]{\left\| #1 \right\|}
\newcommand{\norm}[1]{\left\| #1 \right\|_2}
\def\inner#1#2{\langle #1, #2 \rangle}

\def\eps{{\varepsilon}}

\def\vzero{{\bm{0}}}

\def\vb{{\bm{b}}}

\def\ve{{\bm{e}}}

\def\vu{{\bm{u}}}
\def\vv{{\bm{v}}}

\def\vx{{\bm{x}}}
\def\vy{{\bm{y}}}

\def\gA{{\mathcal{A}}}

\def\gF{{\mathcal{F}}}
\def\gG{{\mathcal{G}}}

\def\gL{{\mathcal{L}}}
\def\gM{{\mathcal{M}}}

\def\gO{{\mathcal{O}}}

\def\sP{{\mathbb{P}}}

\def\sR{{\mathbb{R}}}

\newcommand{\E}{\mathbb{E}}

\newcommand{\Var}{\mathrm{Var}}

\DeclareMathOperator*{\argmin}{arg\,min}

\DeclareMathOperator{\diag}{diag}
\DeclareMathOperator{\spn}{span}
\DeclareMathOperator{\prox}{prox}

\usepackage{amsthm}
\theoremstyle{plain}
\newtheorem{thm}{Theorem}[section]
\newtheorem{defn}[thm]{Definition}

\newtheorem{lemma}[thm]{Lemma}

\newtheorem{remark}[thm]{Remark}
\newtheorem{corollary}[thm]{Corollary}
\newtheorem{prop}[thm]{Proposition}

\usepackage{algorithm}
\usepackage{algorithmic}

\makeatletter
\def\Ddots{\mathinner{\mkern1mu\raise\p@
\vbox{\kern7\p@\hbox{.}}\mkern2mu
\raise4\p@\hbox{.}\mkern2mu\raise7\p@\hbox{.}\mkern1mu}}
\makeatother

\def\pr#1{\sP\left( #1 \right)}

\makeatletter
\newcommand*{\rom}[1]{\expandafter\@slowromancap\romannumeral #1@}
\makeatother
\usepackage{natbib}
\usepackage{caption}
\DeclareMathOperator{\geo}{Geo}

\title{A General Analysis Framework of Lower Complexity Bounds for Finite-Sum Optimization}

\makeatletter
\renewcommand\@date{{%
  \vspace{-\baselineskip}%
  \large\centering
  \begin{tabular}{@{}c@{}}
    Guangzeng Xie \\
    \normalsize smsxgz@pku.edu.cn
  \end{tabular}%
  \quad \quad
  \begin{tabular}{@{}c@{}}
    Luo Luo \\
    \normalsize rickyluoluo@gmail.com
  \end{tabular}
  \quad \quad
  \begin{tabular}{@{}c@{}}
    Zhihua Zhang \\
    \normalsize zhzhang@math.pku.edu.cn
  \end{tabular}

  \bigskip

  School of Mathematical Sciences, Peking University
}}
\makeatother

\newcommand{\B}{\bm{B}}
\newcommand{\A}{\bm{A}}
\newcommand{\I}{\bm{I}}

\usepackage{multirow}

\begin{document}

\maketitle

\begin{abstract}
This paper studies the lower bound complexity for the optimization problem whose objective function is the average of $n$ individual smooth convex functions. We consider the algorithm which gets access to gradient and proximal oracle for each individual component.
For the strongly-convex case, we prove such an algorithm can not reach an $\eps$-suboptimal point in fewer than $\Omega((n+\sqrt{\kappa n})\log(1/\eps))$ iterations, where $\kappa$ is the condition number of the objective function. This lower bound is tighter than previous results and perfectly matches the upper bound of the existing proximal incremental first-order oracle algorithm Point-SAGA.
We develop a novel construction to show the above result, which partitions the tridiagonal matrix of classical examples into $n$ groups. 
This construction is friendly to the analysis of proximal oracle and also could  be used to general convex and average smooth cases naturally.
\end{abstract}

\section{Introduction}
We consider the minimization of the following optimization problem 
\begin{align}
    \min_{\vx\in\sR^d} f(\vx) \triangleq \frac{1}{n}\sum_{i=1}^n f_i(\vx), \label{prob:main}
\end{align}
where the $f_i(\vx)$ are $L$-smooth and $\mu$-strongly convex. 
The condition number is defined as $\kappa=L/\mu$, which  is typically larger than $n$ in real-world applications. Many machine learning models can be formulated as the above problem such as ridge linear regression, ridge logistic regression, smoothed support vector machines, graphical models, etc.
This paper focuses on the first order methods for solving Problem (\ref{prob:main}), which access to the Proximal Incremental First-order Oracle (PIFO) for each individual component, that is,
\begin{align}
    h_f(\vx, i, \gamma) \triangleq \left[ f_i(\vx), \nabla f_i(\vx), \prox^\gamma_{f_i}(\vx) \right],
    \label{oracle}
\end{align}
where $i\in\{1,\dots,n\}$, $\gamma>0$, and the proximal operation is defined as 
\begin{align*}
    \prox^\gamma_{f_i}(\vx) = \argmin_\vu\left\{f_i(\vx) +\frac{1}{2\gamma}\|\vx-\vu\|^2_2\right\}.
\end{align*}
We also define the Incremental First-order Oracle (IFO)
\begin{align*}
    g_f(\vx, i, \gamma) \triangleq \left[ f_i(\vx), \nabla f_i(\vx)\right].
\end{align*}
PIFO provides more information than IFO and it would be potentially more powerful than IFO in first order optimization algorithms. Our goal is to find an $\eps$-suboptimal solution $\hat\vx$ such that \begin{align*}
f(\hat\vx)-\min_{\vx \in \sR^d} f(\vx) \leq \eps
\end{align*}
by using PIFO or IFO.

There are several first-order stochastic algorithms to solve Problem (\ref{prob:main}). The key idea to leverage the structure of $f$ is variance reduction which is effective for ill-conditioned problems. 
For example, SVRG~\citep{zhang2013linear,johnson2013accelerating,DBLP:journals/siamjo/Xiao014} can find an $\eps$-suboptimal solution in $\gO((n+\kappa)\log(1/\eps))$ IFO calls, while the complexity of the classical Nesterov’s acceleration~\citep{nesterov1983method} is $\gO(n\sqrt{\kappa}\log(1/\eps))$.  
Similar results\footnote{
SVRG, SAG and SAGA only need to introduce the proximal operation for composite objective, that is, $f_i(\vx) = g_i(\vx) + h(\vx)$, where $h$ may be non-smooth. 
Their iterations only depend on IFO when all  the $f_i(x)$ are smooth. 
Hence, we regard these algorithms only require IFO calls in this paper.} 
also hold for SAG~\citep{schmidt2017minimizing} and SAGA~\citep{defazio2014saga}.  
In fact, there exists an accelerated stochastic gradient method with $\sqrt{\kappa}$ dependency. \citet{defazio2016simple} introduced a simple and practical accelerated method called Point SAGA, which reduces the iteration complexity to $\gO((n+\sqrt{\kappa n}) \log(1/\eps))$.
The advantage of Point SAGA is in that it has only one parameter to be tuned, but the iteration depends on PIFO rather than IFO. \citet{allen2017katyusha} proposed the Katyusha momentum to accelerate variance reduction algorithms, which achieves the same iteration complexity as Point-SAGA but only requires IFO calls.

The lower bound complexities of IFO algorithms for convex optimization have been well studied \citep{agarwal2015lower,ArjevaniS15Communication,woodworth2016tight,carmon2017lower,lan2017optimal,zhou2019lower}.
\citet{lan2017optimal} showed that at least $\Omega((n {+} \sqrt{\kappa n})\log(1/\eps))$ IFO calls\footnote{
\citeauthor{lan2017optimal}'s construction satisfies $f$ is $\mu$-strongly convex and every $f_i$ is convex, while this paper study the lower bound with stronger condition that is every $f_i$ is $\mu$-strongly convex.
For the same lower bound complexity, the result with stronger assumptions on the objective functions is stronger.} are needed to obtain an $\eps$-suboptimal solution for some complicated objective functions. This lower bound is optimal because it matches the upper bound complexity of Katyusha~\citep{allen2017katyusha}.

It would be interesting  whether we can establish a more efficient PIFO algorithm than IFO one. \citet{woodworth2016tight} provided a lower bound $\Omega(n {+} \sqrt{\kappa n}\log(1/\eps))$  for PIFO algorithms, while the known upper bound of the PIFO algorithm Point SAGA [3] is $\gO((n {+} \sqrt{\kappa n})\log(1/\eps))$. 
The difference of dependency on $n$ implies that the existing theory of PIFO algorithm is not perfect. This gap can not be ignored because the number of components $n$ is typically very large in many machine learning problems. 
A natural question is can we design a PIFO algorithm whose upper bound complexity matches \citeauthor{woodworth2016tight}'s lower bound, or can we improve the lower bound complexity of PIFO to match the upper bound of Point SAGA.

In this paper, we prove the lower bound complexity of PIFO algorithm is $\Omega((n{+} \sqrt{\kappa n}) \log(1/\eps))$ for smooth and strongly-convex $f_i$, which means the existing Point-SAGA~\citep{defazio2016simple} has achieved optimal complexity and PIFO can not lead to a tighter upper bound than IFO. We provide a novel construction, showing the above result by decomposing the classical tridiagonal matrix~\citep{nesterov2013introductory} into $n$ groups. This technique is quite different from the previous lower bound complexity analysis~\citep{agarwal2015lower,woodworth2016tight,lan2017optimal,zhou2019lower}. Moreover, it is very friendly to the analysis of proximal operation and easy to follow. We also use this technique to study general convex and average smooth cases~\citep{allen2018katyushax,zhou2019lower},  obtaining the similar lower bounds to the previous work~\citep{woodworth2016tight,zhou2019lower}. 
In addition, we provide the lower bound complexity of PIFO algorithm for non-convex problem in Appendix \ref{sec:nonconvex} for demonstrating the power of our framework.
And We hope it could be applied in non-smooth problems in future work.


\begin{table}[ht]
\begin{center}
\newcommand{\tabincell}[2]{\begin{tabular}{@{}#1@{}}#2\end{tabular}}
{\small\begin{tabular}{ |c|c|c|c| } 
 \hline
 & Upper Bounds   & Previous Lower Bounds & Our Lower Bounds   \\ \hline
 \tabincell{c}{$f_i$ is $L$-smooth \\ and $\mu$-strongly \\ convex}   
 & \tabincell{c}{\\[-0.15cm]$\gO\left(\left(n + \sqrt{\kappa n}\right)\log(\frac{1}{\eps})\right)$ \\[0.2cm] 
 \citep{allen2017katyusha} \\[0.1cm]  IFO  \\[0.25cm]
 $\gO\left(\left(n + \sqrt{\kappa n}\right)\log(\frac{1}{\eps})\right)$ \\[0.2cm]
 \citep{defazio2016simple} \\[0.1cm]  PIFO \\[0.15cm]}
 & \tabincell{c}{\\[-0.15cm]
$\Omega\left(n + \sqrt{\kappa n}\log(\frac{1}{\eps})\right)$ \\[0.2cm] 
 \!\!\citep{woodworth2016tight}\!\! \\[0.15cm]  PIFO\\[0.2cm]}
 & \tabincell{c}{\\[-0.15cm]$\Omega\left(\left(n + \sqrt{\kappa n}\right)\log(\frac{1}{\eps})\right)$ \\[0.2cm] 
 [Theorem \ref{thm:strongly}] \\[0.1cm]  PIFO\\[0.15cm]} \\\hline
 \tabincell{c}{$f_i$ is $L$-smooth \\  and convex} 
 &  \tabincell{c}{\\[-0.15cm]$\gO\left( n\log(\frac{1}{\eps}) + \sqrt{\frac{nL}{\eps}} \right)$ \\[0.2cm] 
 \citep{allen2017katyusha} \\[0.1cm]  IFO\\[0.15cm]}
 & \tabincell{c}{\\[-0.15cm]$\Omega\left( n + \sqrt{\frac{nL}{\eps}} \right)$ \\[0.2cm] 
 \!\!\citep{woodworth2016tight}\!\!  \\[0.1cm]  PIFO \\[0.15cm]}
 &  \tabincell{c}{\\[-0.15cm]$\Omega\left( n + \sqrt{\frac{nL}{\eps}} \right)$ \\[0.2cm] 
 [Theorem \ref{thm:convex}] \\[0.1cm]  PIFO \\[0.15cm]} \\ \hline
 \tabincell{c}{\\[-0.15cm]\!\!$\{f_i\}_{i=1}^n$ is $L$-average \\ smooth and $f$ is \\ $\mu$-strongly convex\!\! \\[0.2cm]} 
 & \tabincell{c}{\\[-0.3cm]\!\!$\gO\left(\left(n + n^{3/4} \sqrt{\kappa}\right)\log\left( \frac{1}{\eps} \right)\right)$\!\!\\[0.2cm]  
 \citep{allen2018katyushax} \\[0.1cm]  IFO \\[0.15cm]} 
 & \tabincell{c}{\\[-0.15cm]$\Omega\left(n + n^{3/4} \sqrt{\kappa}\log\left( \frac{1}{\eps} \right)\right)$ \\[0.2cm] 
 \citep{zhou2019lower} \\[0.1cm]  IFO\\[0.15cm]}
 & \tabincell{c}{\\[-0.15cm]\!\!$\Omega\left(\left(n + n^{3/4} \sqrt{\kappa}\right)\log\left( \frac{1}{\eps} \right)\right)$\!\! \\[0.2cm] 
 [Theorem \ref{thm:average:strongly}] \\[0.1cm]  PIFO\\[0.15cm]} \\\hline
 \tabincell{c}{\\[-0.3cm]\!\!$\{f_i\}_{i=1}^n$ is $L$-average \\ smooth and $f$ is \\ convex\!\!\\[0.2cm]} 
  & \tabincell{c}{\\[-0.15cm]$\gO\left(n +  n^{3/4} \sqrt{\frac{L}{\eps}} \right)$ \\[0.15cm]
 \citep{allen2018katyushax} \\[0.1cm]  IFO \\[0.2cm]}
  & \tabincell{c}{\\[-0.15cm]$\Omega\left(n +  n^{3/4} \sqrt{\frac{L}{\eps}} \right)$ \\[0.2cm] 
 \citep{zhou2019lower} \\[0.1cm]  IFO\\[0.15cm]}
  & \tabincell{c}{\\[-0.15cm]$\Omega\left(n +  n^{3/4} \sqrt{\frac{L}{\eps}} \right)$ \\[0.15cm] 
 [Theorem \ref{thm:average:convex}] \\[0.1cm]  PIFO\\[0.2cm]} \\\hline
\end{tabular}}
\caption{We compare our PIFO lower bounds with existing results of IFO or PIFO algorithms, where $\kappa=L/\mu$. Note that the call of PIFO could obtain more information than IFO.
Hence, any PIFO lower bound also can be regarded as an IFO lower bound, not vice versa. }\label{table:UL}
\end{center}
\end{table}

\begin{table}[ht]
\begin{center}
\newcommand{\tabincell}[2]{\begin{tabular}{@{}#1@{}}#2\end{tabular}}
{\begin{tabular}{ |c|c|c|c| } 
 \hline
 & Previous Lower Bounds & Our Lower Bounds   \\ \hline
 \tabincell{c}{$f_i$ is $L$-smooth \\ and $\mu$-strongly \\ convex}   
 & \tabincell{c}{\\[-0.15cm] 
 $\#{\rm PIFO}=\Omega\left(n + \sqrt{\kappa n}\log(\frac{1}{\eps})\right)$ \\[0.2cm] 
 $d=\gO\left(\frac{\kappa n}{\eps}\log^5\left(\frac{1}{\eps}\right)\right)$  \\[0.2cm]
 \citep{woodworth2016tight}\\[0.15cm]}
 & \tabincell{c}{\\[-0.15cm]$\#{\rm PIFO}=\Omega\left(\left(n + \sqrt{\kappa n}\right)\log(\frac{1}{\eps})\right)$ \\[0.2cm] 
 $d=\gO\left(\sqrt{\frac{\kappa}{n}}\log\left(\frac{1}{\eps}\right)\right)$ \\[0.2cm]
 [Theorem \ref{thm:strongly}]\\[0.15cm]} \\\hline
 \tabincell{c}{$f_i$ is $L$-smooth \\  and convex} 
 & \tabincell{c}{\\[-0.15cm]$\#{\rm PIFO}=\Omega\left( n + \sqrt{\frac{nL}{\eps}} \right)$ \\[0.2cm] 
 $d=\gO\left(\frac{L^2}{\eps^2}\log\left(\frac{1}{\eps}\right)\right)$  \\[0.2cm]
 \citep{woodworth2016tight}\\[0.15cm]}
 &  \tabincell{c}{\\[-0.15cm]$\#{\rm PIFO}=\Omega\left( n + \sqrt{\frac{nL}{\eps}} \right)$ \\[0.2cm] 
 $d=\gO\left(1+\sqrt{\frac{L}{n\eps}}\right)$ \\[0.2cm]
 [Theorem \ref{thm:convex}] \\[0.15cm]} \\ \hline
 \tabincell{c}{\\[-0.3cm]$\{f_i\}_{i=1}^n$ is $L$-average \\ smooth and $f$ is \\ $\mu$-strongly convex \\[0.15cm]} 
 & \tabincell{c}{\\[-0.15cm]$\#{\rm IFO}=\Omega\left(n + n^{3/4} \sqrt{\kappa}\log\left( \frac{1}{\eps} \right)\right)$ \\[0.2cm] 
 $d=\gO\left(n+n^{3/4}\sqrt{\kappa}\log\left(\frac{1}{\eps}\right)\right)$ \\[0.2cm]
 \citep{zhou2019lower}\\[0.15cm]}
 & \tabincell{c}{\\[-0.15cm]$\#{\rm PIFO}=\Omega\left(\left(n + n^{3/4} \sqrt{\kappa}\right)\log\left( \frac{1}{\eps} \right)\right)$ \\[0.2cm] 
 $d=\gO\left(n^{-1/4}\sqrt{\kappa}\log\left(\frac{1}{\eps}\right)\right)$ \\[0.2cm]
 [Theorem \ref{thm:average:strongly}] \\[0.15cm]} \\\hline
 \tabincell{c}{\\[-0.15cm]$\{f_i\}_{i=1}^n$ is $L$-average \\ smooth and $f$ is \\ convex\\[0.1cm]} 
  & \tabincell{c}{\\[-0.3cm]$\#{\rm IFO}=\Omega\left(n +  n^{3/4} \sqrt{\frac{L}{\eps}} \right)$ \\[0.2cm] 
  $d=\gO\left(n+n^{3/4}\sqrt{\frac{L}{\eps}}\right)$ \\[0.2cm]
 \citep{zhou2019lower} \\[0.15cm]}
  & \tabincell{c}{\\[-0.15cm]$\#{\rm PIFO}=\Omega\left(n +  n^{3/4} \sqrt{\frac{L}{\eps}} \right)$ \\[0.2cm]
  $d=\gO\left(1+n^{-1/4}\sqrt{\frac{L}{\eps}}\right)$ \\[0.2cm]
 [Theorem \ref{thm:average:convex}]\\[0.15cm]} \\\hline
\end{tabular}}
\caption{We compare our PIFO lower bounds with previous results, including the number of PIFO or IFO calls to obtain $\eps$-suboptimal point and the required number of dimensions in corresponding construction.}\label{table:DL}
\end{center}
\end{table}

\section{A General Analysis Framework}\label{sec:pre}

In this paper, we consider the Proximal Incremental First-order Oracle (PIFO) algorithm for smooth convex finite-sum optimization. 
All the omitted proof in this section can be found in Appendix \ref{appendix:sec:pre} and Appendix \ref{sec:geo} for a detailed version.
We analyze the lower bounds of the algorithms when the objective functions are respectively strongly convex, general convex, smooth and average smooth~\citep{zhou2019lower}. 

\begin{defn}
For any differentiable function $f:\sR^m \rightarrow \sR$,
\begin{itemize}
    \item $f$ is convex, if for any $\vx,\vy\in\sR^m$ it satisfies 
    $f(\vy)\geq f(\vx) + \langle \nabla f(\vx), \vy-\vx \rangle$.
    \item $f$ is $\mu$-strongly convex, if for any $\vx,\vy\in\sR^m$ it satisfies \\[0.15cm]
    $f(\vy)\geq f(\vx) + \langle \nabla f(\vx), \vy-\vx \rangle + \dfrac{\mu}{2}\|\vx-\vy\|_2^2$.
    \item $f$ is $L$-smooth, if for any $\vx,\vy\in\sR^m$ it satisfies 
    $\|\nabla f(\vx) - \nabla f(\vy)\|_2 \leq L\| \vx - \vy \|_2$.
\end{itemize}
\end{defn}

\begin{defn}
We say differentiable functions $\{f_i\}_{i=1}^n, ~f_i: \sR^m \rightarrow \sR$ to be $L$-average smooth if for any $\vx, \vy \in \sR^m$, they satisfy  
\begin{align}
    \frac{1}{n} \sum_{i=1}^n \norm{\nabla f_i(\vx) - \nabla f_i(\vy)}^2 \le L^2 \norm{\vx - \vy}^2.
\end{align}
\end{defn}

\begin{remark}\label{remark:smooth}
We point out that 
\begin{enumerate}
    \item if each $f_i$ is $L$-smooth, then we have $\{f_i\}_{i=1}^n$ is $L$-average smooth.
    \item if $\{f_i\}_{i=1}^n$ is $L$-average smooth, then we have $f(\vx) = \frac{1}{n} \sum_{i=1}^n f_i(\vx)$ is $L$-smooth.
\end{enumerate}\hskip1cm
\end{remark}
We present the formal definition for PIFO algorithm.
\begin{defn}
Consider a stochastic optimization algorithm $\gA$ to solve Problem~(\ref{prob:main}).
Let $\vx_t$ be the point obtained at time-step $t$ and the algorithm starts with $\vx_0$.
The algorithm $\gA$ is said to be a PIFO algorithm if for any $t\geq 0$, we have
\begin{align}
    \vx_t \in \spn\big\{\vx_{0},\dots,\vx_{t-1}, \nabla f_{i_1} (\vx_{0}), \cdots, \nabla f_{i_t} (\vx_{t-1}), \prox_{f_{i_1}}^{\gamma_1} (\vx_{0}), \cdots,  \prox_{f_{i_t}}^{\gamma_t} (\vx_{t-1}) \big\},
\end{align}
where $i_t$ is a random variable supported on $[n]$ and takes 
\begin{align}
\sP(i_t = j) = p_j,
\end{align}
for each $t \ge 0$ and $1 \le j \le n$ where $\sum_{j=1}^n p_j = 1$. 
\end{defn}
Without loss of generality, we assume that $\vx_0=\vzero$ and $p_1 \le p_2 \le \cdots \le p_n$ to simplify our analysis. Otherwise, we can take $\{{\hat f}_i(\vx) = f_i(\vx + \vx_0)\}_{i=1}^n$ into consideration. On the other hand, suppose that $p_{s_1} \le p_{s_2} \le \cdots \le p_{s_n}$ where $\{s_i\}_{i=1}^n$ is a permutation of $[n]$. Define $\{{\tilde f}_i\}_{i=1}^n$ such that ${\tilde f}_{s_i} = f_{i}$, then $\gA$ takes component ${\tilde f}_{s_i}$ by probability $p_{s_i}$, i.e., $\gA$ takes component $f_{i}$ by probability $p_{s_i}$. 

To demonstrate the construction of adversarial functions, we first introduce the following class of matrices:
\begin{align*}
    \B(m, \omega) = 
\begin{bmatrix}
    & & & -1 & 1 \\
    & & -1 & 1 & \\
    & \Ddots & \Ddots & & \\
    -1 & 1 & & & \\
    \omega & & & & 
\end{bmatrix} 
    \in \sR^{m \times m}. 
\end{align*}
Then we define
\begin{align}
\A(m, \omega) \triangleq \B(m, \omega)^{\top} \B(m, \omega) = 
\begin{bmatrix}
    \omega^2 + 1 & -1 & & & \\
    -1 & 2 & -1 & & \\
    & \ddots & \ddots & & \\
    & & -1 & 2 & -1 \\
    & & & -1 & 1 
\end{bmatrix} \label{ABTB}.
\end{align}
The matrix $\A(m, \omega)$ is widely-used in the analysis of lower bounds for convex optimization \citep{nesterov2013introductory,agarwal2015lower,lan2017optimal,carmon2017lower,zhou2019lower}. We now present a decomposition of $\A(m, \omega)$ based on Eq.~(\ref{ABTB}).

Denote the $l$-th row of the matrix $\B(m, \omega)$ by $\vb_l(m, \omega)^{\top}$ 
and  let
\[\gL_i = \big\{ l: 1 \le l \le m, l \equiv i - 1 (\bmod~n) \big\},\quad i = 1, 2, \cdots, n.
\]
Our construction is based on  the following class of functions
\begin{align*}
    r(\vx; \lambda_0, \lambda_1, \lambda_2, m, \omega) \triangleq \frac{1}{n}\sum_{i=1}^n r_i(\vx; \lambda_0, \lambda_1, \lambda_2, m, \omega),
\end{align*}
where
{\small
\begin{align}\label{defn-f}
\!r_i(\vx; \lambda_0, \lambda_1, \lambda_2, m, \omega)=
\begin{cases}
\lambda_1 \sum\limits_{l \in \gL_1} \norm{\vb_{l}(m, \omega)^{\top}\vx}^2 + \lambda_2 \norm{\vx}^2 - \lambda_0 \inner{\ve_m}{\vx}, & \text{ for } i = 1, \\
\lambda_1\sum\limits_{l \in \gL_i} \norm{\vb_{l}(m, \omega)^{\top} \vx}^2 + \lambda_2 \norm{\vx}^2, & \text{ for } i = 2, 3, \cdots, n.
\end{cases}
\end{align}}

We can determine the smooth and strongly-convex coefficients of $r_i$ as follows.
\begin{prop}\label{prop:base}
For any $\lambda_1 > 0, \lambda_2 \ge 0, \omega<\sqrt{2}$, we have that the $r_i$ are $(4 \lambda_1 + 2 \lambda_2)$-smooth and $\lambda_2$-strongly convex, and $\{r_i\}_{i=1}^n$ is $L'$-average smooth where
\begin{align*}
    L' = 2\sqrt{\frac{4}{n} \left[ (\lambda_1 + \lambda_2)^2 + \lambda_1^2 \right] + \lambda_2^2}.
\end{align*}
\end{prop}

We define the subspaces $\{\gF_k\}_{k=0}^m$ where 
\begin{align*}
\gF_k = \begin{cases}
\spn \{ \ve_m, \ve_{m-1}, \cdots, \ve_{m-k+1}\}, & \text{for } 1 \le k \le m, \\
 \{\vzero\}, & \text{for } k=0.
\end{cases}
\end{align*}

The following technical lemma plays a crucial role in our proof.
\begin{lemma}\label{lem:jump}
For any $\lambda_0 \neq 0, \lambda_1 > 0, \lambda_2 \ge 0$ and $\vx \in \gF_k$, $0 \le k < m$, 
we have that
\begin{align*}
    \nabla r_{i}(\vx;\lambda_0, \lambda_1, \lambda_2, m, \omega)\; \mbox{ and }  ~\prox_{r_{i}}^{\gamma} (\vx)\in 
    \begin{cases}
    \gF_{k+1}, & \text{ if } k \equiv i - 1 (\bmod ~n), \\
    \gF_{k}, & \text{ otherwise}.
    \end{cases}
\end{align*}
\end{lemma}
In short, if $\vx \in \gF_k$ and let $f_i(\vx)\triangleq r_i(\vx;\lambda_0, \lambda_1, \lambda_2,\omega)$, then there exists only one $i\in\{1,\dots,n\}$ such that $h_f(\vx, i, \gamma)$ could (and only could) provide additional information in $\gF_{k+1}$.
The ``only one'' property is important to the lower bound analysis for  first order stochastic optimization algorithms~\citep{lan2017optimal,zhou2019lower}, but these prior constructions only work for IFO rather than PIFO.

Lemma \ref{lem:jump} implies that $\vx_t = \vzero$ will host until algorithm $\gA$ draws the component $f_1$.
Then, for any $t < T_1 = \min_t \{t: i_t = 1\}$, we have $\vx_t \in \gF_0$ and $\vx_{T_1} \in \gF_1$. The value of $T_1$ can be regarded as the smallest integer such that $\vx_{T_1}$ could host.
Similarly, we can define $T_k$ to be the smallest integer such that $\vx_{T_k} \in \gF_k$ could host. We give the formal definition of $T_k$ recursively and connect it to geometrically distributed random variables in the following corollary.

\begin{corollary}\label{coro:stopping-time}
Let
\begin{align}\label{def:stoppong-time}
    T_0 = 0,~\text{ and } 
    ~T_k = \min_t \{t: t > T_{k-1}, i_t \equiv k~(\bmod ~n)\}~\text{ for } k \ge 1.
\end{align}
Then for any $k \ge 1$ and $t < T_{k}$, we have $\vx_t \in \gF_{k-1}$. 
Moreover, $T_k$ can be written as sum of $k$ independent random variables $\{Y_l\}_{1 \le l \le k}$, i.e.,
\begin{align*}
    T_k = \sum_{l=1}^{k} Y_l,
\end{align*}
where $Y_l$ follows a geometric distribution with success probability $q_l = p_{l'}$ where $l' \equiv l~(\bmod~n), 1 \le l' \le n$. 
\end{corollary}


The basic idea of our analysis is that we guarantee the minimizer of $r$ lies in $\gF_m$
and assure the PIFO algorithm extend the space of $\spn\{\vx_0,\vx_1,\dots,\vx_t\}$ slowly with $t$ increasing. 
We know that $\spn\{\vx_0,\vx_1,\dots,\vx_{T_k}\} \subseteq \gF_{k}$ by Corollary \ref{coro:stopping-time}.   
Hence, $T_k$ is just the quantity that reflects how  $\spn\{\vx_0,\vx_1,\dots,\vx_t\}$ verifies. Because  $T_k$ can be written as the sum of geometrically distributed random variables, we needs to introduce some properties of such random variables which derive the lower bounds of our construction.

\begin{lemma}\label{lem:geo}
Let $\{Y_i\}_{1 \le i \le N}$ be independent random variables, 
and $Y_i$ follows a geometric distribution with success probability $p_i$.
Then
\begin{align}
    \pr{\sum_{i=1}^N Y_i > \frac{N^2}{4(\sum_{i=1}^N p_i)}} \ge 1 - \frac{16}{9N}.
\end{align}
\end{lemma}
From Lemma \ref{lem:geo}, the following result implies how many PIFO calls we need.
\begin{lemma}\label{lem:base}
If $M \ge 1$ satisfies $\min_{\vx \in \gF_M} f(\vx) - \min_{\vx \in \sR^m} f(\vx) \ge 
9\eps$ and $N = n(M+1)/4$, 
then we have
\begin{align*}
    \min_{t \le N} \E f(\vx_t) - \min_{\vx \in \sR^m} f(\vx) \ge \eps.
\end{align*}
\end{lemma}

\begin{proof}
Denote $\min_{\vx \in \sR^m} f(\vx)$ by $f^*$. For $t \le N$, we have
\begin{align*}
    \E f(\vx_{t}) - f^* &\ge \E [f(\vx_{t}) - f^* | N < T_{M+1}] \pr{N < T_{M+1}}\\
    &\ge \E [\min_{\vx \in \gF_{M}} f(\vx) - f^* | N < T_{M+1}] \pr{N < T_{M+1}} \\
    &\ge 9\eps \pr{T_{M+1} > N},
\end{align*}
where $T_{M+1}$ is defined in (\ref{def:stoppong-time}), and the second inequality follows from Corollary \ref{coro:stopping-time} (if $N < T_{M+1}$, then $\vx_t \in \gF_M$ for $t \le N$).

By Corollary \ref{coro:stopping-time}, $T_{M+1}$ can be written as $T_{M+1} = \sum_{l=1}^{M+1} Y_l$,
where $\{Y_l\}_{1 \le l \le M+1}$ are independent random variables, and $Y_l$ follows a geometric distribution with success probability $q_l = p_{l'}$ ($l' \equiv l (\bmod ~n)$, $1 \le l' \le n$).

Recalling that $p_1 \le p_2 \le \cdots \le p_n$, we have 
$$\sum_{l=1}^{M+1} q_l \le \frac{M+1}{n}.$$
Therefore, by Lemma \ref{lem:geo}, we have
$$\pr{\sum_{l=1}^{M+1} Y_l > \frac{(M+1)n}{4}} \ge 1 - \frac{16}{9(M+1)} \ge \frac{1}{9},$$
that is, 
\[ \pr{T_{M+1} > N} \ge \frac{1}{9}.\]
Hence,
\begin{align*}
    \E f(\vx_{N}) - f^* \ge 9\eps \pr{T_{M+1} > N} \ge \eps.
\end{align*}

\end{proof}

\begin{remark}
In fact, a more strong conclusion hosts:
\begin{align*}
    \E \left[\min_{t \le N} f(\vx_t)\right] - \min_{\vx \in \sR^m} f(\vx) \ge \eps.
\end{align*}
\end{remark}

\section{Main Results}\label{sec:mainresults}

We present the our lower bound results for PIFO algorithms and summarize all of results in Table~\ref{table:UL} and \ref{table:DL} . We first start with smooth and strongly convex setting, then consider the general convex and average smooth cases.

\begin{thm}\label{thm:strongly}
For any PIFO algorithm $\gA$ and any $L, \mu, n, \Delta, \eps$ such that
$\kappa=L/\mu \ge n/2 + 1$, and $\eps/\Delta \le 0.00327$,
there exist a dimension $d = \gO \left( \sqrt{\kappa/n} \log \left(\Delta/\eps\right)\right)$ and $n$ $L$-smooth and $\mu$-strongly convex functions $\{f_i:\sR^d\rightarrow\sR\}_{i=1}^n$  such that $f(\vx_0) - f(\vx^*) = \Delta$. In order to find $\hat{\vx} \in \sR^d$ such that $\E f(\hat{\vx}) - f(\vx^{*}) < \eps$, $\gA$ needs at least $\Omega\left(\left(n {+} \sqrt{\kappa n}\right)\log\left(\Delta/\eps \right)\right)$ queries to $h_f$.
\end{thm}

\begin{remark}
In fact, the upper bound of the existing PIFO algorithm Point SAGA \citep{defazio2016simple}~
\footnote{
\citet{defazio2016simple} proves Point SAGA requires $\gO\left(\left(n + \sqrt{\kappa n}\right)\log\left(1/\eps \right)\right)$ PIFO calls to find $\hat\vx$ such that $\E\|{\hat\vx}-\vx^*\|_2^2<\eps$, where $\vx^*=\arg\min_\vx f(\vx)$, which is not identical to the condition $\E f({\hat\vx})-f(\vx^*)<\eps\|\vx_0-\vx^*\|_2^2$ in Theorem \ref{thm:strongly}. 
However, it is unnecessary to worry about it because we also establish a PIFO lower bound $\Omega\left(\left(n + \sqrt{\kappa n}\right)\log\left(1/\eps \right)\right)$ for $\E\|{\hat\vx}-\vx^*\|_2^2<\eps\|\vx_0-\vx^*\|_2^2$ in Theorem \ref{thm:strongly:example:2}.} is 
$\gO\left(\left(n + \sqrt{\kappa n}\right)\log\left(1/\eps \right)\right)$.
Hence, the lower bound in Theorem \ref{thm:strongly} is tight, 
while \citet{woodworth2016tight} only provided lower bound $\Omega\left(n {+} \sqrt{\kappa n}\log\left(1/\eps \right)\right)$ which is not optimal to $n$ dependency.
The theorem also shows that the PIFO algorithm can not be more powerful than the IFO algorithm in the worst case, because the upper bound of the IFO algorithm~\citep{allen2017katyusha} is also $\gO\left(\left(n {+} \sqrt{\kappa n}\right)\log\left(1/\eps \right)\right)$.
\end{remark}

Next we give the lower bound when the objective function is not strongly-convex.
\begin{thm}\label{thm:convex}
For any PIFO algorithm $\gA$ and any $L, n, B, \eps$ such that
$\eps \le L B^2 /4$,
there exist a dimension $d = \gO \left( 1 + B \sqrt{L/(n\eps)} \right)$ and $n$ $L$-smooth and convex functions $\{f_i:\sR^d\rightarrow\sR\}_{i=1}^n$   such that $\norm{\vx_0 - \vx^*} \le B$. In order to find $\hat{\vx} \in \sR^d$ such that $\E f(\hat{\vx}) - f(\vx^{*}) < \eps$, $\gA$ needs at least $\Omega\left(n {+} B\sqrt{nL/\eps} \right)$ queries to $h_f$.
\end{thm}

\begin{remark}
The lower bound in Theorem \ref{thm:convex} is the same as the one of \citeauthor{woodworth2016tight}'s result. However, our construction only requires the dimension be $\gO \left( 1 + B \sqrt{L/(n\eps)} \right)$, which is much smaller than 
$\gO \left(\frac{L^2B^4}{\eps^2}\log\left(\frac{nLB^2}{\eps}\right)\right)$ in
\citep{woodworth2016tight}.  
\end{remark}

Then we extend our results to the weaker assumption: that is, the objective function $F$ is $L$-average smooth~\citep{zhou2019lower}. We start with the case that $F$ is strongly convex.

\begin{thm}\label{thm:average:strongly}
For any PIFO algorithm $\gA$ and any $L, \mu, n, \Delta, \eps$ such that 
$\kappa=L/\mu \ge \sqrt{3/n} \left(\frac{n}{2} + 1\right)$, and $\eps/\Delta \le 0.00327$,
there exist a dimension  $d = \gO \left( n^{-1/4} \sqrt{\kappa} \log \left(\Delta/\eps\right)\right)$ and $n$ functions $\{f_i:\sR^d\rightarrow\sR\}_{i=1}^n$ where  the $\{f_i\}_{i=1}^n$ are $L$-average smooth and $f$ is $\mu$-strongly convex,  such that $f(\vx_0) - f(\vx^*) = \Delta$. In order to find $\hat{\vx} \in \sR^d$ such that $\E f(\hat{\vx}) - f(\vx^{*}) < \eps$, $\gA$ needs at least $\Omega\left(\left(n {+} n^{3/4} \sqrt{\kappa}\right)\log\left( \Delta/\eps \right)\right)$ queries to $h_f$.
\end{thm}

\begin{remark}
Compared with \citeauthor{zhou2019lower}'s lower bound $\Omega\left(n + n^{3/4} \sqrt{\kappa}\log\left( \Delta/\eps \right)\right)$ for IFO algorithms,
 Theorem \ref{thm:average:strongly} shows tighter dependency on $n$ and supports PIFO algorithms additionally.
\end{remark}

We also give the lower bound for general convex case under the $L$-average smooth condition.
\begin{thm}\label{thm:average:convex}
For any PIFO algorithm $\gA$ and any $L, n, B, \eps$ such that
$\eps \le L B^2 /4 $,
there exist a dimension $d = \gO \left( 1 + B n^{-1/4} \sqrt{L/\eps} \right)$ and $n$ functions $\{f_i:\sR^d\rightarrow\sR\}_{i=1}^n$ which the  $\{f_i\}_{i=1}^n$ are $L$-average smooth and $f$ is convex, such that $\norm{\vx_0 - \vx^*} \le B$. In order to find $\hat{\vx} \in \sR^d$ such that $\E f(\hat{\vx}) - f(\vx^{*}) < \eps$, $\gA$ needs at least $\Omega\left(n + B n^{3/4} \sqrt{L/\eps} \right)$ queries to $h_f$.
\end{thm}

\begin{remark}
The lower bound in Theorem \ref{thm:average:convex} is comparable to the one of \citeauthor{zhou2019lower}'s result, but our construction only requires the dimension be $\gO \left( 1 + B n^{-1/4} \sqrt{L/\eps} \right)$, which is much smaller than 
$\gO \left(n+Bn^{3/4}\sqrt{L/\eps}\right) $ in \citep{zhou2019lower}.
\end{remark}

\section{Constructions in Proof of Main Theorems}
We demonstrate the detailed constructions for PIFO lower bounds in this section. 
All the omitted proof in this section can be found in Appendix for a detailed version.
\subsection{Strongly Convex Case}\label{sec:strongly}

The analysis of lower bound complexity for the strongly-convex case depends on the following construction.

\begin{defn}\label{defn:sc}
For fixed $L, \mu, \Delta, n$, 
let $\alpha = \sqrt{ \frac{2(L/\mu - 1)}{n} + 1}$.
We define $f_{\text{SC}, i} : \sR^m \rightarrow \sR$ as follows
\begin{align}
    f_{\text{SC}, i} (\vx) = r_i\left(\vx; \sqrt{\frac{2(L - \mu)n\Delta}{\alpha - 1}}, \frac{L-\mu}{4}, \frac{\mu}{2}, m, \sqrt{\frac{2}{\alpha + 1}} \right), \text{ for } 1 \le i \le n,
\end{align}
and 
\begin{align*}
    F_{\text{SC}}(\vx) \triangleq \frac{1}{n} \sum_{i=1}^n f_{\text{SC}, i}(\vx) 
    = \frac{L-\mu}{4n} \norm{\B\left(m, \sqrt{\frac{2}{\alpha + 1}}\right) \vx}^2 + \frac{\mu}{2} \norm{\vx}^2 
    - \sqrt{\frac{2(L-\mu)\Delta}{n(\alpha-1)}}\inner{\ve_m}{\vx}.
\end{align*}
\end{defn}

\begin{prop}\label{prop:strongly}
For any $n \ge 2$, $m \ge 2$, $f_{\text{SC}, i}$ and $F_{\text{SC}}$ in Definition \ref{defn:sc} satisfy:
\begin{enumerate}
    \item $f_{\text{SC}, i}$ is $L$-smooth and $\mu$-strongly convex. \label{scp1}
    \item The minimizer of the function $F_{\text{SC}}$ is
    $$\vx^{*} = \argmin_{\vx \in \sR^m}F_{\text{SC}}(\vx) = \sqrt{\frac{2\Delta n(\alpha+1)^2}{(L-\mu)(\alpha-1)}} (q^{m}, q^{m-1}, \cdots, q)^{\top},$$ 
    where $q = \frac{\alpha - 1}{\alpha + 1}$.
    Moreover, $F_{\text{SC}}(\vx^*) = -\Delta$. \label{scp2}
    \item For $1 \le k \le m - 1$, we have
    \begin{align}\label{prop:strongly:3}
    \min_{\vx \in \gF_k} F_{\text{SC}}(\vx) - F_{\text{SC}}(\vx^{*}) \ge \Delta q^{2k}. 
    \end{align}\label{scp3}
\end{enumerate}
\end{prop}

Note that the $f_{\text{SC}, i}$ are $L$-smooth and $\mu$-strongly convex, and $F_{\text{SC}}(\vx_0) - F_{\text{SC}} (\vx^*) = \Delta$.
Next we show that the functions $\{f_{\text{SC}, i}\}_{i=1}^n$ are ``hard enough'' for any PIFO algorithm $\gA$, and deduce the conclusion of Theorem \ref{thm:strongly}.

\begin{thm}\label{thm:strongly:example}
Suppose that
\begin{align*}
    \frac{L}{\mu} \ge \frac{n}{2} + 1, ~ \eps \le \frac{\Delta}{9} \left(\frac{\sqrt{2}-1}{\sqrt{2}+1}\right)^{2}, \text{ and } 
    m = \frac{1}{4}\left(\sqrt{2\frac{L/\mu - 1}{n} + 1}\right) \log \left(\frac{\Delta}{9\eps}\right) + 1.
\end{align*}
In order to find $\hat{\vx} \in \sR^m$ such that $\E F_{\text{SC}}(\hat{\vx}) - F_{\text{SC}}(\vx^*) < \eps$, PIFO algorithm $\gA$ needs at least $\Omega\left(\left(n + \sqrt{\frac{nL}{\mu}}\right)\log\left( \frac{\Delta}{\eps} \right)\right)$ queries to $h_{F_{\text{SC}}}$.
\end{thm}

\begin{proof}
Let $M = \floor{\frac{\log (9\eps/\Delta)}{2\log q}}$, then we have 
\begin{align*}
    \argmin_{\vx \in \gF_M} F_{\text{SC}}(\vx) - F_{\text{SC}}(\vx^*) \ge \Delta q^{2M} \ge 9 \eps,
\end{align*}
where the first inequality is according to the third property of Proposition \ref{prop:strongly}.

Following from Lemma \ref{lem:base}, for $M \ge 1$ and $N = (M+1)n/4$, we have
\begin{align*}
    \min_{t \le N} \E F_{\text{SC}}(\vx_t) - F_{\text{SC}}(\vx^*) \ge \eps.
\end{align*}
Therefore, in order to find $\hat{\vx} \in \sR^m$ such that $\E F_{\text{SC}}(\hat{\vx}) - F_{\text{SC}}(\vx^{*}) < \eps$, $\gA$ needs at least $N$ queries to $h_{F_{\text{SC}}}$.

Next, observe that function $h(\beta) = \frac{1}{\log \left(\frac{\beta + 1}{\beta - 1}\right)} - \frac{\beta}{2}$ is increasing when $\beta>1$
and $L/\mu \ge n/2 + 1$, $\alpha = \sqrt{2 \frac{L/\mu - 1}{n} + 1} \ge \sqrt{2}$.
Thus, we have
\begin{align*}
    -\frac{1}{\log(q)} &= \frac{1}{\log \left(\frac{\alpha + 1}{\alpha - 1}\right)} 
    \ge \frac{\alpha}{2} + h(\sqrt{2}) \\
    &= \frac{1}{2}\sqrt{2 \frac{L/\mu - 1}{n} + 1} + h(\sqrt{2}) \\
    &\ge \frac{\sqrt{2}}{4} \left( \sqrt{2\frac{L/\mu - 1}{n}} + 1 \right) + h(\sqrt{2}) \\
    &\ge \frac{1}{2} \sqrt{\frac{L/\mu - 1}{n}} + \frac{\sqrt{2}}{4} + h(\sqrt{2}),
\end{align*}
and 
\begin{align*}
    N &= (M+1)n/4 = \frac{n}{4} \left(\floor{\frac{\log (9\eps/\Delta)}{2\log q}} + 1\right) \\
    &\ge \frac{n}{8} \left( -\frac{1}{\log(q)} \right) \log\left( \frac{\Delta}{9\eps} \right) \\
    &\ge \frac{n}{8}\left( \frac{1}{2} \sqrt{\frac{L/\mu - 1}{n}} + \frac{\sqrt{2}}{4} + h(\sqrt{2})  \right) \log\left( \frac{\Delta}{9\eps} \right) \\
    &= \Omega\left( \left( n + \sqrt{\frac{n L}{\mu}} \right) \log\left( \frac{\Delta}{9\eps} \right) \right)
\end{align*}

At last, we must to ensure that $1 \le M < m$, that is
\begin{align}
    1 \le \frac{\log (9\eps/\Delta)}{2\log q} < m.
\end{align}
Note that $\lim_{\beta \rightarrow +\infty} h(\beta) = 0$, so $- 1/ \log(q) \le \alpha /2$.
Thus the above conditions are satisfied when 
\begin{align*}
    m = \frac{\log (\Delta / (9\eps))}{2 (-\log q)} + 1 \le
    \frac{1}{4}\left(\sqrt{2\frac{L/\mu - 1}{n} + 1}\right) \log \left(\frac{\Delta}{9\eps}\right) + 1
    = \gO \left( \sqrt{\frac{L}{n\mu}} \log \left(\frac{\Delta}{\eps}\right)\right),
\end{align*}
and 
\begin{align*}
    \frac{\eps}{\Delta} \le \frac{1}{9} \left(\frac{\sqrt{2}-1}{\sqrt{2}+1}\right)^{2} \le \frac{1}{9} \left(\frac{\alpha-1}{\alpha+1}\right)^{2}.
\end{align*}

\end{proof}

\citet{defazio2016simple} showed that the PIFO algorithm Point SAGA has the convergence result $\E \norm{\vx_t - \vx^*}^2 \le (q')^t\norm{\vx_0 - \vx^*}$, 
where $q'$ satisfies $-1/\log(q') = \gO \left(n + \sqrt{n L/\mu}\right)$.
To match this form of upper bound, we point out that a similar scheme of lower bound holds for $\{f_{\text{SC}, i}\}_{i=1}^n$.

\begin{thm}\label{thm:strongly:example:2}
Suppose that
\begin{align*}
    \frac{L}{\mu} \ge \frac{n}{2} + 1, ~ \eps \le \frac{1}{18} \left(\frac{\sqrt{2}-1}{\sqrt{2}+1}\right)^{2}, \text{ and } 
    m = \frac{1}{2}\left(\sqrt{2\frac{L/\mu - 1}{n} + 1}\right) \log \left(\frac{1}{18\eps}\right) + 1.
\end{align*}
In order to find $\hat{\vx} \in \sR^m$ such that $\E \norm{\hat{\vx} - \vx^*}^2< \eps\norm{\vx_0 - \vx^*}^2$, PIFO algorithm $\gA$ needs at least $\Omega\left(\left(n + \sqrt{\frac{nL}{\mu}}\right)\log\left( \frac{1}{\eps} \right)\right)$ queries to $h_{F_{\text{SC}}}$.
\end{thm}

\begin{proof}
Denote $\xi = \sqrt{\frac{2 \Delta n (\alpha + 1)^2}{(L - \mu) (\alpha - 1)}}$, and $M = \floor{\frac{\log(18\eps)}{2\log q}}$. \\
For $1 \le M \le m/2$, $N = n(M+1)/4$ and $t \le N$, we have
\begin{align*}
    \E \norm{\vx_t - \vx^*}^2 &\ge \E \left[ \norm{\vx_t - \vx^*}^2 \bigg\vert N < T_{M+1} \right] \pr{N < T_{M+1}} \\
    &\ge \E \left[ \min_{\vx \in \gF_M} \norm{\vx - \vx^*}^2 \bigg\vert N < T_{M+1} \right] \pr{N < T_{M+1}} \\
    &\ge \frac{1}{9} \min_{\vx \in \gF_M} \norm{\vx - \vx^*}^2.
\end{align*}
where $T_{M+1}$ is defined in (\ref{def:stoppong-time}), the second inequality follows from Corollary \ref{coro:stopping-time} (if $N < T_{M+1}$, then $\vx_t \in \gF_M$ for $t \le N$), and the last inequality is established because of Corollary \ref{coro:stopping-time} (More detailed explanation refer to our proof of Lemma \ref{lem:base}).

By Proposition \ref{prop:strongly}, we know that $\vx^* = \xi (q^m, q^{m-1}, \cdots, q)^{\top}$, and 
\begin{align*}
    \norm{\vx_0 - \vx^*}^2 = \norm{\vx^*}^2 = \xi^2 \frac{q^{2} - q^{2(m+1)}}{1 - q^2}.
\end{align*}

Note that if $\vx \in \gF_M$, then $x_1 = x_2 = \cdots = x_{m-M} = 0$, thus
\begin{align*}
    \min_{\vx \in \gF_M} \norm{\vx - \vx^*}^2 = \xi^2 \sum_{l = m-M}^m q^{2(m-l+1)} = \xi^2 \frac{q^{2(M+1)} - q^{2(m+1)}}{1 - q^2}.
\end{align*}

Thus, for $t \le N$ and $M \le m/2$, we have
\begin{align*}
    \frac{\E \norm{\vx_t - \vx^*}^2}{\norm{\vx_t - \vx^*}^2} &\ge \frac{1}{9} \frac{q^{2M} - q^{2m}}{1 - q^{2m}} \\
    &\ge \frac{1}{18} q^{2M} = \frac{1}{18} q^{2\floor{\frac{\log(18\eps)}{2\log q}}} \ge \eps,
\end{align*}
where the second inequality is due to 
\begin{align*}
    \frac{q^{2M} - q^{2m}}{1 - q^{2m}} - \frac{q^{2M}}{2} &= \frac{q^{2M} - 2 q^{2m} + q^{2(m+M)}}{2(1-q^{2m})} \\
    &= \frac{q^{2M}}{2(1-q^{2m})} (1 - 2 q^{2(m-M)} + q^{2m}) \\
    &\ge \frac{q^{2M}}{2(1-q^{2m})} (1 - 2 q^{m} + q^{2m}) \ge 0.
\end{align*}

Therefore, in order to find $\hat{\vx} \in \sR^m$ such that $\frac{\E \norm{\hat{\vx} - \vx^*}^2}{\norm{\vx_0 - \vx^*}^2} < \eps$, $\gA$ needs at least $N$ queries to $h_{F_{\text{SC}}}$.

As we have showed in proof of Theorem \ref{thm:strongly:example}, for $L/\mu \ge n/2 + 1$, we have
\begin{align*}
    \frac{1}{2}\sqrt{2\frac{L/\mu - 1}{n} + 1} \ge -\frac{1}{\log(q)} \ge c_1 \left( \sqrt{\frac{L/\mu - 1}{n}} + 1 \right),
\end{align*}
and 
\begin{align*}
    N &= \frac{n}{4} (M+1) \ge \frac{n}{4} \frac{\log(18\eps)}{2\log q} \\
    &\ge \frac{c_1}{8} \left( n + \sqrt{n(L/\mu - 1)} \right) \log \left( \frac{1}{18\eps} \right) \\
    &= \Omega \left( \left(n + \sqrt{\frac{nL}{\mu}}\right) \log \left( \frac{1}{\eps} \right) \right).
\end{align*}

At last, we have to ensure that $1 \le M \le m/2$, that is 
\begin{align*}
    1 \le \frac{\log(18\eps)}{2\log q} < m/2.
\end{align*}

The above conditions are satisfied when 
\begin{align*}
    m = \frac{\log(1/(18\eps))}{-\log q} + 1 \le \frac{1}{2}\left(\sqrt{2\frac{L/\mu - 1}{n} + 1}\right) \log \left(\frac{1}{18\eps}\right) + 1 
    = \gO \left( \sqrt{\frac{L}{n\mu}} \log \left(\frac{1}{\eps}\right)\right),
\end{align*}
and 
\begin{align*}
    \eps \le \frac{1}{18} q^{2}.
\end{align*}
Observe that when $L/\mu \le n / 2 + 1$, we have $\alpha \ge \sqrt{2}$ and $q = \frac{\alpha - 1}{\alpha + 1} \ge \frac{\sqrt{2} - 1}{\sqrt{2} + 1}$.
Hence, we just need $\eps \le \frac{1}{18}\left(\frac{\sqrt{2} - 1}{\sqrt{2} + 1}\right)^2 \approx 0.00164$.

\end{proof}
\subsection{Convex Case}\label{sec:convex}
The analysis of lower bound complexity for non strongly-convex cases depends on the following construction.
\begin{defn}\label{defn:c}
For fixed $L, B, n$, we define $f_{\text{C}, i}: \sR^m \rightarrow \sR$ as follows
\begin{align}
    f_{\text{C}, i}(\vx) = r_i \left(\vx; \frac{\sqrt{3}}{2} \frac{B L}{(m+1)^{3/2}}, \frac{L}{4}, 0, m, 1 \right)
\end{align}
and 
\begin{align*}
    F_{\text{C}} (\vx) \triangleq  \frac{1}{n} \sum_{i=1}^n f_{\text{C}, i}(\vx) = \frac{L}{4n} \norm{\B(m, 1) \vx}^2 - \frac{\sqrt{3}}{2}\frac{BL}{(m+1)^{3/2}n}\inner{\ve_m}{\vx}.
\end{align*}
\end{defn}

\begin{prop}\label{prop:convex}
For any $n \ge 2$, $m \ge 2$, following properties hold:
\begin{enumerate}
    \item $f_{\text{C}, i}$ is $L$-smooth and convex. 
    \item The minimizer of the function $F_{\text{C}}$ is
    $$\vx^{*} = \argmin_{\vx \in \sR^m}F_{\text{C}}(\vx) = \frac{2\xi}{L} \left(1, 2, \cdots, m\right)^{\top},$$ 
    where $\xi = \frac{\sqrt{3}}{2} \frac{BL}{(m+1)^{3/2}}$.
    Moreover, $F_{\text{C}}(\vx^*) = - \frac{m \xi^2}{n L}$ and $\norm{\vx_0 - \vx^*}^2 \le B^2$.
    \item For $1 \le k \le m$, we have
    \begin{align}\label{prop:convex:3}
    \min_{\vx \in \gF_k} F_{\text{C}}(\vx) - F_{\text{C}}(\vx^{*}) = \frac{\xi^2}{n L} (m - k).
    \end{align}
\end{enumerate}
\end{prop}

Note that the $f_{\text{C}, i}$ are $L$-smooth and convex, and $\norm{\vx_0 - \vx^*} \le B$. 
Next we establish the lower bound for functions $f_{\text{C}, i}$ defined above.
\begin{thm}\label{thm:convex:example}
Suppose that 
\begin{align*}
    \eps \le \frac{B^2 L}{384 n}\; \mbox{ and } \; m = \floor{\sqrt{\frac{B^2 L}{24 n \eps}}} - 1.
\end{align*}
In order to find $\hat{\vx} \in \sR^m$ such that $\E F_{\text{C}}(\hat{\vx}) - F_{\text{C}}(\vx^*) < \eps$, $\gA$ needs at least $\Omega\left(n + B \sqrt{\frac{n L}{\eps}}\right)$ queries to $h_{F_{\text{C}}}$.
\end{thm}

\begin{proof}
Since $\eps \le \frac{B^2 L}{384 n}$, we have $m \ge 3$.
Let $\xi = \frac{\sqrt{3}}{2} \frac{BL}{(m+1)^{3/2}}$.

For $M = \floor{\frac{m-1}{2}} \ge 1$, we have $m - M \ge (m+1) /2$, and
\begin{align*}
    \min_{\vx \in \gF_M} F_{\text{C}}(\vx) - F_{\text{C}}(\vx^{*}) &= \frac{\xi^2}{n L} (m - M) = \frac{3 B^2 L}{4n} \frac{m-M}{(m+1)^3} \\ &\ge \frac{3 B^2 L}{8n} \frac{1}{(m+1)^2} \ge 9\eps,
\end{align*}
where the first equation is according to the 3rd property in Proposition \ref{prop:convex} and the last inequality follows from $m + 1 \le B \sqrt{L/(24n\eps)}$.

Similar to the proof of Theorem \ref{thm:strongly:example}, by Lemma \ref{lem:base},  we have
\begin{align*}
    \min_{t \le N} \E F_{\text{C}}(\vx_t) - F_{\text{C}}(\vx^*) \ge \eps.
\end{align*}
In other words, in order to find $\hat{\vx} \in \sR^m$ such that $\E F_{\text{C}}(\hat{\vx}) - F_{\text{C}}(\vx^{*}) < \eps$, $\gA$ needs at least $N$ queries to $h_F$.

At last, observe that
\begin{align*}
    N &= (M+1)n/4 = \frac{n}{4} \floor{\frac{m+1}{2}} \\
    &\ge \frac{n(m-1)}{8} \\
    &\ge \frac{n}{8} \left(\sqrt{\frac{B^2 L}{24 n \eps}} - 2\right) \\
    &= \Omega \left( n + B\sqrt{\frac{n L}{\eps}} \right),
\end{align*}
where we have recalled $\eps \le \frac{B^2 L}{384 n}$ in last equation.

\end{proof}

To derive Theorem \ref{thm:convex}, we also need the following lemma for the case $\eps > \frac{B^2 L}{384n}$.
\begin{lemma}\label{lem:convex:example:2}
For any PIFO algorithm $\gA$ and any $L, n, B, \eps$ such that $\eps \le LB^2 /4$, there exist $n$ $L$-smooth and convex  functions $\{f_i: \sR \rightarrow \sR\}_{i=1}^n$  such that $|x_0 - x^*| \le B$. In order to find $\hat{x} \in \sR$ such that $\E F(\hat{x}) - F(x^*) < \eps$, $\gA$ needs at least $\Omega(n)$ queries to $h_F$.
\end{lemma}

\begin{proof}
Consider the following functions $\{g_i\}_{1 \le i \le n}$, $g_i: \sR \rightarrow \sR$, where
\begin{align*}
    g_1(x) &= \frac{L}{2} x^2 - n L B x, \\
    g_i(x) &= \frac{L}{2} x^2, \\
    G(x) &= \frac{1}{n} \sum_{i=1}^n g_i(x) = \frac{L}{2} x^2 - L B x .
\end{align*}

First observe that
\begin{align*}
    x^* = \argmin_{x \in \sR} G(x) = B, \\
    G(0) - G(x^*) = \frac{LB^2}{2},
\end{align*}
and $|x_0 - x^*| = B$.

For $i > 1$, we have $\frac{d g_i(x)}{d x} \vert_{x = 0} = 0$ and $\prox_{g_i}^{\gamma} (0) = 0$.
Thus $x_t = 0$ will host till our first-order method $\gA$ draws the component $f_1$.
That is, for $t < T = \argmin \{t: i_t = 1\}$, we have $x_t = 0$.

Hence, for $t \le \frac{1}{2p_1}$, we have
\begin{align*}
    \E G(x_t) - F(x^*) &\ge \E \left[G(x_t) - G(x^*) \Big\vert \frac{1}{2p_1} < T\right] \pr{\frac{1}{2p_1} < T} \\
    &= \frac{LB^2}{2} \pr{\frac{1}{2p_1} < T}.
\end{align*}

Note that $T$ follows a geometric distribution with success probability $p_1 \le 1/n$, and
\begin{align*}
    \pr{ T > \frac{1}{2p_1} } = \pr{ T > \floor{\frac{1}{2p_1}} } = (1 - p_1)^{\floor{\frac{1}{2p_1}}} \\
    \ge (1 - p_1)^{\frac{1}{2p_1}} \ge (1 - 1/n)^{n/2} \ge \frac{1}{2},
\end{align*}
where the second inequality follows from $h(z) = \frac{\log(1 - z)}{2z}$ is a decreasing function.

Thus, for $t \le \frac{1}{2p_1}$, we have
\begin{align*}
    \E G(x_t) - F(x^*) \ge \frac{LB^2}{4} \ge \eps
\end{align*}

Thus, in order to find $\hat{x} \in \sR$ such that $\E F(\hat{x}) - F(x^{*}) < \eps$, $\gA$ needs at least $ \frac{1}{2p_1} \ge n/2 = \Omega\left( n \right)$ queries to $h_G$.

\end{proof}

It is worth noting that if $\eps > \frac{B^2 L}{384 n}$, then $\Omega (n) = \Omega \left( n + B\sqrt{\frac{nL}{\eps}} \right)$. 
Thus combining Theorem \ref{thm:convex:example} and Lemma \ref{lem:convex:example:2}, we obtain Theorem \ref{thm:convex}.
\subsection{Average Smooth Case}\label{sec:average}
\citet{zhou2019lower} established lower bounds of IFO complexity under the average smooth assumption. Here we  demonstrate that our technique can also develop lower bounds of PIFO algorithm under this assumption.



\subsubsection{$F$ is Strongly Convex}

For fixed $L', \mu, \Delta, n, \eps$, we set $L = \sqrt{\frac{n(L'^2 - \mu^2)}{2} - \mu^2}$, and consider $\{f_{\text{SC}, i}\}_{i=1}^n$ and $F_{\text{SC}}$ defined in Definition \ref{defn:sc}.

\begin{prop}\label{prop:average:example:strongly}
For $n \ge 2$, we have that
\begin{enumerate}
    \item $F_{\text{SC}}(\vx)$ is $\mu$-strongly convex and $\{f_{\text{SC}, i}\}_{i=1}^n$ is $L'$-average smooth.
    \item If $\frac{L'}{\mu} \ge \sqrt{\frac{3}{n}}(\frac{n}{2} + 1)$, 
    then we have $\sqrt{\frac{n}{3}} L' \le L \le \sqrt{\frac{n}{2}} L'$ and $L / \mu \ge n/2 + 1$.
\end{enumerate}
\end{prop}

\begin{proof}
~
\begin{enumerate}
\item It is easily to check that $F_{\text{SC}}(\vx)$ is $\mu$-strongly convex.
Following from Proposition \ref{prop:base}, then $\{f_{\text{SC}, i}\}_{i=1}^n$ is $\hat{L}$-average smooth, 
where
\begin{align*}
    \hat{L} 
    &= \sqrt{\frac{16}{n} \left[ \left(\frac{L+\mu}{4}\right)^2 + \left(\frac{L-\mu}{4}\right)^2 \right] + \mu^2} \\
    &= \sqrt{\frac{2(L^2 + \mu^2)}{n} + \mu^2} = L'.
\end{align*}

\item Clearly, $L = \sqrt{\frac{n(L'^2 - \mu^2)}{2} - \mu^2} \le \sqrt{\frac{n}{2}} L'$. \\
Furthermore, according to $\frac{L'}{\mu} \ge \sqrt{\frac{3}{n}}(\frac{n}{2} + 1)$, we have
\begin{align*}
    L^2 - \frac{n}{3}L'^2 &= \frac{n}{2} (L'^2 - \mu^2) - \mu^2 - \frac{n}{3}L'^2 \\
    &= \frac{1}{2} \left(\frac{n}{2} + 1\right)^2 \mu^2 - \frac{n + 2}{2} \mu^2 \\
    &= \left(\frac{n^2}{8} - \frac{1}{2}\right)\mu^2 \ge 0,
\end{align*}
and, $L/\mu \ge \sqrt{\frac{n}{3}} L' / \mu \ge n / 2 + 1$.

\end{enumerate}
\end{proof}

\begin{thm}\label{thm:average:example:strongly}
Suppose that
\begin{align*}
    \frac{L'}{\mu} \ge \sqrt{\frac{3}{n}}\left(\frac{n}{2} + 1\right), ~ \eps \le \frac{\Delta}{9} \left(\frac{\sqrt{2}-1}{\sqrt{2}+1}\right)^{2}, \text{ and } 
    m = \frac{1}{4}\left(\sqrt{ \sqrt{\frac{2}{n}} \frac{L'}{\mu} + 1}\right) \log \left(\frac{\Delta}{9\eps}\right) + 1.
\end{align*}
In order to find $\hat{\vx} \in \sR^m$ such that $\E F_{\text{SC}}(\hat{\vx}) - F_{\text{SC}}(\vx^*) < \eps$, PIFO algorithm $\gA$ needs at least $\Omega\left(\left(n + n^{3/4} \sqrt{\frac{L'}{\mu}}\right)\log\left( \frac{\Delta}{\eps} \right)\right)$ queries to $h_{F_{\text{SC}}}$.
\end{thm}

\begin{proof}
By 2nd property of Proposition \ref{prop:average:example:strongly}, we know that $L / \mu \ge n/2 + 1$.
Moreover, 
\begin{align*}
    m &= \frac{1}{4}\left(\sqrt{ \sqrt{\frac{2}{n}} \frac{L'}{\mu} + 1}\right) \log \left(\frac{\Delta}{9\eps}\right) + 1 \\
    &\ge \frac{1}{4}\left(\sqrt{2\frac{L/\mu - 1}{n} + 1}\right) \log \left(\frac{\Delta}{9\eps}\right) + 1,
\end{align*}

Then, by Theorem \ref{thm:strongly:example} \footnote{By the proof of Theorem \ref{thm:strongly:example}, a larger dimension $m$ does not affect the conclusion of the theorem.}, in order to find $\hat{\vx} \in \sR^m$ such that $\E F_{\text{SC}}(\hat{\vx}) - F_{\text{SC}}(\vx^{*}) < \eps$, $\gA$ needs at least $N$ queries to $h_{F_{\text{SC}}}$, where
\begin{align*}
    N &= \Omega\left(\left(n + \sqrt{\frac{nL}{\mu}}\right)\log\left( \frac{\Delta}{\eps} \right)\right) \\
    &= \Omega\left(\left(n + \sqrt{\frac{n\sqrt{n/3} L'}{\mu}}\right)\log\left( \frac{\Delta}{\eps} \right)\right) \\
    &= \Omega\left(\left(n + n^{3/4} \sqrt{\frac{L'}{\mu}}\right)\log\left( \frac{\Delta}{\eps} \right)\right).
\end{align*}

\end{proof}

\subsubsection{$F$ is Convex}
For fixed $L', B, n, \eps$, we set $L = \sqrt{\frac{n}{2}}L'$, and consider $\{f_{\text{C}, i}\}_{i=1}^n$ and $F_{\text{C}}$ defined in Definition \ref{defn:c}.
It follows from Proposition \ref{prop:base}  that $\{f_{\text{C}, i}\}_{i=1}^n$ is $L'$-average smooth.

\begin{thm}\label{thm:average:example:convex}
Suppose that 
\begin{align*}
    \eps \le \frac{\sqrt{2}}{768} \frac{B^2 L'}{\sqrt{n}}\; \mbox{ and } \; m = \floor{ \frac{\sqrt[4]{18}}{12} B n^{-1/4}\sqrt{\frac{L'}{\eps}}} - 1.
\end{align*}
In order to find $\hat{\vx} \in \sR^m$ such that $\E F_{\text{C}}(\hat{\vx}) - F_{\text{C}}(\vx^*) < \eps$, $\gA$ needs at least $\Omega\left(n + B n^{3/4} \sqrt{\frac{L'}{\eps}}\right)$ queries to $h_{F_{\text{C}}}$.
\end{thm}

\begin{proof}
Note that
\begin{align*}
    \eps &\le \frac{\sqrt{2}}{768} \frac{B^2 L'}{\sqrt{n}} = \frac{B^2 L}{384 n}, \\
    m &= \floor{ \frac{\sqrt[4]{18}}{12} B n^{-1/4}\sqrt{\frac{L'}{\eps}}} - 1 = \floor{ \sqrt{\frac{B^2 L}{24 n \eps}}} - 1.
\end{align*}
By Theorem \ref{thm:convex:example}, in order to find $\hat{\vx} \in \sR^m$ such that $\E F_{\text{C}}(\hat{\vx}) - F_{\text{C}}(\vx^{*}) < \eps$, $\gA$ needs at least $N$ queries to $h_{F_{\text{C}}}$, where
\begin{align*}
    N &= \Omega\left(n + B \sqrt{\frac{n L}{\eps}}\right) \\
    &= \Omega\left(n + B \sqrt{\frac{n \sqrt{n/2}L'}{\eps}}\right) \\
    &= \Omega\left(n + B n^{3/4} \sqrt{\frac{L'}{\eps}}\right).
\end{align*}
\end{proof}

Similar to Lemma \ref{lem:convex:example:2}, we also need the following lemma for the case $\eps > \frac{\sqrt{2}}{768} \frac{B^2 L'}{\sqrt{n}}$.
\begin{lemma}\label{lem:average:example:2}
For any PIFO algorithm $\gA$ and any $L, n, B, \eps$ such that $\eps \le LB^2 /4$, there exist n functions $\{f_i: \sR \rightarrow \sR\}_{i=1}^n$ which is $L$-average smooth,  such that $F(x)$ is convex and $\|x_0 - x^*\|_2 \le B$. In order to find $\hat{x} \in \sR$ such that $\E F(\hat{x}) - F(x^*) < \eps$, $\gA$ needs at least $\Omega(n)$ queries to $h_F$.
\end{lemma}

\begin{proof}
Note that $\{g_i\}_{i=1}^n$ defined in above proof is also $L$-average smooth, so Lemma \ref{lem:average:example:2} hosts for the same reason.
\end{proof}

Similarly, note that if $\eps > \frac{\sqrt{2}}{768} \frac{B^2 L'}{\sqrt{n}}$, then $\Omega(n) = \Omega\left(n + B n^{3/4} \sqrt{\frac{L'}{\eps}}\right) $. In summary, we obtain  Theorem \ref{thm:average:convex}.

\section{Conclusion and Future Work}

In this paper we have studied lower bound of PIFO algorithm for smooth finite-sum optimization. We have given a tight lower bound of PIFO algorithms in the strongly convex case.
We have proposed a novel construction framework that is very useful to the analysis of proximal algorithms.
Based on this framework, We have also extended our result to non-strongly convex, average smooth problems and non-convex problems (see Appendix \ref{sec:nonconvex}). It would be interesting to prove tight lower bounds of proximal algorithms for non-smooth problems in future work.

\paragraph{Acknowledgements:} We thank Dachao Lin and Yuze Han for their helpful discussions about Lemma \ref{lem:geo}.

\bibliographystyle{plainnat}
\bibliography{reference}

\newpage
\appendix
\section{Detailed Proof for Section \ref{sec:pre}}\label{appendix:sec:pre}
In this section, we use $\Norm{\A}$ to denote the spectral radius of $\A$.

For simplicity, let
\begin{align*}
    \B = \B(m, \omega) = 
\begin{bmatrix}
    & & & -1 & 1 \\
    & & -1 & 1 & \\
    & \Ddots & \Ddots & & \\
    -1 & 1 & & & \\
    \omega & & & & 
\end{bmatrix} 
    \in \sR^{m \times m}, 
\end{align*}
$\vb_l^{\top}$ is the $l$-th row of $\B$, and
$f_i(\vx) = r_i(\vx; \lambda_0, \lambda_1, \lambda_2, m, \omega)$.

Recall that
$$\gL_i = \{ l: 1 \le l \le m, l \equiv i - 1 (\bmod n) \}, i = 1, 2, \cdots, n.$$ 
For $1 \le i \le n$, let $\B_i$ be a submatrix which is formed from rows $\gL_i$ of $\B$, that is
\begin{align*}
    \B_i = \B[\gL_i;]
\end{align*}

Then $f_i$ can be wriiten as
\begin{align*}
f_i(\vx) &= \lambda_1 \norm{\B_i \vx}^2 + \lambda_2 \norm{\vx}^2 - \eta_i \inner{\ve_m}{\vx},
\end{align*}
where $\eta_1 = \lambda_0, \eta_i = 0, i \ge 2$.

\begin{proof}[\textbf{Proof of Proposition \ref{prop:base}}]
Note that 
\begin{align*}
    \inner{\vu}{\B_i^{\top}\B_i \vu} &= \norm{\B_i \vu}^2 \\
    &= \sum_{l \in \gL_i} (\vb_l^{\top} \vu)^2 \\\
    &= 
    \begin{cases}
    \sum_{l \in \gL_i\backslash \{m\}} (u_{m-l} - u_{m-l+1})^2 + \omega^2 u_m^2  ~~(\text{if } m \in \gL_i) \\
    \sum_{l \in \gL_i} (u_{m-l} - u_{m-l+1})^2 
    \end{cases} \\
    &\le 2 \norm{\vu}^2,
\end{align*}
where the last inequality is according to $(x + y)^2 \le 2(x^2 + y^2)$, and $|l_1 - l_2| \ge n \ge 2$ for $l_1, l_2 \in \gL_i$. \\
Hence, $\Norm{\B_i^{\top}\B_i} \le 2$, and
\begin{align*}
    \Norm{\nabla^2 f_i(\vx)} = \Norm{2 \lambda_1 \B_i^{\top} \B_i + 2 \lambda_2 \I} \le 4 \lambda_1 + 2 \lambda_2.
\end{align*}

\vskip 10pt
Next, observe that
\begin{align*}
    \norm{\nabla f_i(\vx) - \nabla f_i(\vy)}^2 
    = \norm{(2\lambda_1 \B_i^{\top} \B_i + 2\lambda_2 \I) (\vx - \vy)}^2
\end{align*}
Let $\vu = \vx - \vy$.\\
Note that
\begin{align*}
    \vb_l \vb_l^{\top} \vu = 
    \begin{cases}
    (u_{m-l} - u_{m-l+1}) (\ve_{m-l} - \ve_{m-l+1}), &l < m, \\
    \omega^2 u_1 \ve_1, &l = m.
    \end{cases}
\end{align*}

Thus, if $m \notin \gL_i$, then 
\begin{align*}
    &\quad \norm{(2\lambda_1 \B_i^{\top} \B_i + 2\lambda_2 \I) \vu}^2 \\
    &= \norm{2\lambda_1 \sum_{l \in \gL_i} (u_{m-l} - u_{m-l+1}) (\ve_{m-l} - \ve_{m-l+1}) + 2\lambda_2 \vu}^2 \\
    &= \sum_{m-l \in \gL_i} \left[(2\lambda_1 (u_{l} - u_{l+1}) + 2\lambda_2 u_{l})^2
    + (-2\lambda_1 (u_{l} - u_{l+1}) + 2\lambda_2 u_{l+1})^2 \right]
    + \sum_{\substack{m-l \notin \gL_i \\ m-l+1 \notin \gL_i}} (2\lambda_2 u_{l})^2 \\
    &\le \sum_{m-l \in \gL_i} 8\left[ (\lambda_1 + \lambda_2)^2 + \lambda_1^2 \right](u_l^2 + u_{l+1}^2) 
    + 4\lambda_2^2 \norm{\vu}^2.
\end{align*}

Similarly, if $m \in \gL_i$, then

\begin{align*}
    &\quad \norm{(2\lambda_1 \B_i^{\top} \B_i + 2\lambda_2 \I) \vu}^2 \\
    &\le \sum_{\substack{m-l \in \gL_i \\ l \neq 0}} 8\left[ (\lambda_1 + \lambda_2)^2 + \lambda_1^2 \right](u_l^2 + u_{l+1}^2) 
    + 4(\lambda_1 \omega^2 + \lambda_2)^2 u_1^2
    +  4\lambda_2^2 \norm{\vu}^2.
\end{align*}

Therefore, we have
\begin{align*}
    &\quad \frac{1}{n} \sum_{i=1}^n \norm{\nabla f_i(\vx) - \nabla f_i(\vy)}^2 \\
    &\le \frac{1}{n} \left[ \sum_{l=1}^{m-1} 8\left[ (\lambda_1 + \lambda_2)^2 + \lambda_1^2 \right](u_l^2 + u_{l+1}^2) 
    + 4(2\lambda_1  + \lambda_2)^2 u_1^2 \right] + 4 \lambda_2^2 \norm{\vu}^2 \\
    &\le \frac{16}{n} \left[ (\lambda_1 + \lambda_2)^2 + \lambda_1^2 \right] \norm{\vu}^2 + 4 \lambda_2^2 \norm{\vu}^2,
\end{align*}
where we have used $(2\lambda_1 + \lambda_2)^2 \le 2 \left[ (\lambda_1 + \lambda_2)^2 + \lambda_1^2 \right]$.

In summary, we get that $\{f_i\}_{1 \le i \le n}$ is $L'$-average smooth, where
\begin{align*}
    L' = 2\sqrt{\frac{4}{n} \left[ (\lambda_1 + \lambda_2)^2 + \lambda_1^2 \right] + \lambda_2^2}.
\end{align*}

\end{proof}

\begin{proof}[\textbf{Proof of Lemma \ref{lem:jump}}]
For $\vx \in \gF_k ~(k \ge 1)$, we have 
\begin{align*}
    \vb_l^{\top} \vx &= 0 \text{ for } l > k, \\
    \vb_l &\in \gF_k \text{ for } l < k, \\ 
    \vb_k &\in \gF_{k+1}.
\end{align*}
Consequently, for $l \neq k$, $\vb_l \vb_l^{\top} \vx = (\vb_l^{\top} \vx) \vb_l \in \gF_k $,
and $\vb_k \vb_k^{\top} \vx \in \gF_{k+1} $.

For $k = 0$, we have $\vx = \vzero$, and
\begin{align*}
\nabla f_1(\vx) = \lambda_0 \ve_m \in \gF_1, \\
\nabla f_j(\vx) = \vzero ~(j \ge 2).
\end{align*}

Moreover, we suppose $k \ge 1$, $k \in \gL_{i}$. Since
\begin{align*}
    \nabla f_j(\vx) &= 2 \lambda_1 \B_j^{\top} \B_j \vx + 2\lambda_2 \vx - \eta_j \ve_m \\
    &= 2\lambda_1 \sum_{l \in \gL_j} \vb_l^{\top} \vb_l \vx + 2\lambda_2 \vx - \eta_j \ve_m.
\end{align*}
Hence, $\nabla f_{i} (\vx) \in \gF_{k+1}$ and $\nabla f_j (\vx) \in \gF_{k} ~(j \neq i)$.

\vskip 5pt
Now, we turn to consider $\vu = \prox_{f_j}^{\gamma} (\vx)$. We have
\begin{align*}
    \left(2 \lambda_1 \B_j^{\top} \B_j + \left(2 \lambda_2 + \frac{1}{\gamma}\right)\I  \right) \vu = \eta_j \ve_m + \frac{1}{\gamma} \vx,
\end{align*}
i.e.,
\begin{align*}
    \vu = c_1 (\I + c_2 \B_j^{\top} \B_j)^{-1} \vy,
\end{align*}
where $c_1 = \frac{1}{2\lambda_2 + 1/\gamma}$, $c_2 = \frac{2\lambda_1}{2\lambda_2 + 1/\gamma}$, and $\vy = \eta_j \ve_m + \frac{1}{\gamma} \vx$. 

Note that 
\begin{align*}
    (\I + c_2 \B_j^{\top} \B_j)^{-1} = \I - \B_j^{\top} \left( \frac{1}{c_2}\I + \B_j \B_j^{\top} \right)^{-1} \B_j.
\end{align*}
If $k = 0$ and $j > 1$, we have $\vy = \vzero$ and $\vu = \vzero$. \\
If $k = 0$ and $j = 1$, we have $\vy = \lambda_0 \ve_m$. On this case, $\B_1 \ve_m = \vzero$, so $\vu = c_1 \vy \in \gF_1$. 

For $k \ge 1$, we know that $\vy \in \gF_k$.
And observe that if $|l - l'| \ge 2$, then $\vb_{l}^{\top} \vb_{l'} = 0$, and consequently $\B_j \B_j^{\top}$ is a diagonal matrix, 
so we can assume that $\frac{1}{c_2}\I + \B_j \B_j^{\top} = \diag(\beta_{j,1}, \cdots, \beta_{j, |\gL_j|})$.
Therefore,
\begin{align*}
    \vu = c_1 \vy - c_1 \sum_{s = 1}^{|\gL_j|} \beta_{j, s} \vb_{l_{j, s}} \vb_{l_{j, s}}^{\top} \vy, 
\end{align*}
where we assume that $\gL_j = \{l_{j, 1}, \cdots, l_{j, |\gL_j|}\}$.

Thus, we have $\prox_{f_{i}}^{\gamma} (\vx) \in \gF_{k+1}$ for $k \in \gL_i$
and $\prox_{f_{j}}^{\gamma} (\vx) \in \gF_{k} ~(j \neq i)$.

\end{proof}

\begin{proof}[\textbf{Proof of Corollary \ref{coro:stopping-time}}]
Denote 
\[\spn \{ \nabla f_{i_1}(\vx_0), \cdots, \nabla f_{i_t}(\vx_{t-1}), \prox_{f_{i_1}}^{\gamma_1} (\vx_0), \cdots, \prox_{f_{i_t}}^{\gamma_t} (\vx_{t-1}) \}\] 
by $\gM_t$.
We know that $\vx_t \in \gM_t$.

Suppose that $\gM_T \subseteq \gF_{k-1}$ for some $T$ and let $T' = \argmin{t: t > T, i_t \equiv k (\bmod ~n) }$.

By Lemma \ref{lem:jump}, for $T < t < T'$, 
we can use a simple induction to obtain that 
\[\spn\{ \nabla f_{i_{t}}(\vx_{t-1}), \prox_{f_{i_{t}}}^{\gamma_t} (\vx_{t-1}) \} \subseteq \gF_{k-1}\]
and $\gM_{t} \subseteq \gF_{k-1}$. 

Moreover, since $i_{T'} \equiv k (\bmod ~n)$, we have
\[\spn\{ \nabla f_{i_{T'}}(\vx_{T'-1}), \prox_{f_{i_{T'}}}^{\gamma_{T'}} (\vx_{T'-1}) \} \subseteq \gF_{k}\]
and $\gM_{T'} \subseteq \gF_{k}$.

Following from above statement, it is easily to check that for $t < T_k$, we have $\vx_t \in \gM_t \subseteq \gF_{k-1}$.

Next, note that 
\begin{align*}
    &\quad \pr{T_{k} - T_{k-1} = s} \\
    &= \pr{ i_{T_{k-1} + 1} \not\equiv k (\bmod ~n), \cdots,  i_{T_{k-1} + s - 1} \not\equiv k (\bmod ~n), i_{T_{k-1} + s} \equiv k (\bmod ~n)} \\
    &= \pr{ i_{T_{k-1} + 1} \neq k', \cdots,  i_{T_{k-1} + s - 1} \neq k', i_{T_{k-1} + s} = k'} \\
    &= (1 - p_{k'})^{s-1} p_{k'},
\end{align*}
where $k' \equiv k (\bmod ~n), 1 \le k' \le n$. So $T_k - T_{k-1}$ is a geometric random variable with success probability $p_{k'}$.

On the other hand, $T_k - T_{k-1}$ is just dependent on $i_{T_{k-1} + 1}, \cdots, i_{T_k}$, thus for $l \neq k$, $T_l - T_{l-1}$ is independent with $T_k - T_{k-1}$.

Therefore, 
\begin{align*}
    T_k = \sum_{l=1}^k (T_l - T_{l-1}) = \sum_{i=1}^k Y_l,
\end{align*}
where $Y_l$ follows a geometric distribution with success probability $q_l = p_{l'}$ where $l' \equiv l (\bmod n), 1 \le l' \le n$. 

\end{proof}

\begin{proof}[\textbf{Proof of Remark \ref{remark:smooth}}]
If each $f_i$ is $L$-smooth, then for any $\vx, \vy \in \sR^m$ we have
\begin{align*}
    \norm{\nabla f_i(\vx) - \nabla f_i(\vy)}^2 \le L^2 \norm{\vx - \vy}^2,
\end{align*}
and consequently,
\begin{align}
    \frac{1}{n} \sum_{i=1}^n \norm{\nabla f_i(\vx) - \nabla f_i(\vy)}^2 \le L^2 \norm{\vx - \vy}^2.
\end{align}

\vskip 10pt
If $\{f_i\}_{i=1}^n$ is $L$-average smooth, then for any $\vx, \vy \in \sR^m$ we have
\begin{align*}
    \norm{\nabla f(\vx) - \nabla f(\vy)}^2 &= \frac{1}{n^2} \norm{\sum_{i=1}^n (\nabla f_i(\vx) - \nabla f_i(\vy))}^2 \\
    &\le \frac{1}{n^2} \left(\sum_{i=1}^n \norm{\nabla f_i(\vx) - \nabla f_i(\vy)} \right)^2 \\
    &\le \frac{1}{n} \sum_{i=1}^n \norm{\nabla f_i(\vx) - \nabla f_i(\vy)}^2 \\
    &\le L^2 \norm{\vx - \vy}^2.
\end{align*}

\end{proof}
\section{Results about Sum of Geometric Distributed Random Variables}\label{sec:geo}
\begin{lemma}
Let $X_1 \sim \geo(p_1), X_2 \sim \geo(p_2)$ be independent random variables. 
For any positive integer $j$, if $p_1 \neq p_2$, then 
\begin{align}\label{pr-neq}
\pr{X_1 + X_2 > j} = \frac{p_2 (1 - p_1)^j - p_1 (1 - p_2)^j}{p_2 - p_1},
\end{align}
and if $p_1 = p_2$, then 
\begin{align}\label{pr-eq}
\pr{X_1 + X_2 > j} = j p_1(1 - p_1)^{j-1} + (1 - p_1)^j.
\end{align}
\end{lemma}
\begin{proof}
\begin{align*}
    \pr{X_1 + X_2 > j}
    &= \sum_{l=1}^{j} \pr{X_1 = l} \pr{X_2 > j - l} 
    + \pr{X_1 > j} \\
    &= \sum_{l=1}^{j} (1-p_1)^{l-1} p_1 (1-p_2)^{j-l} + (1-p_1)^j \\
    &= p_1 (1-p_2)^{j-1} \sum_{l=1}^{j} \left( \frac{1-p_1}{1-p_2} \right)^{l-1}
    + (1-p_1)^j
\end{align*}

Thus if $p_1 = p_2$, $\pr{X_1 + X_2 > j} = j p_1(1 - p_1)^{j-1} + (1 - p_1)^j$.

For $p_1 \neq p_2$, 
\begin{align*}
    \pr{X_1 + X_2 > j}
    &= p_1 \frac{(1 - p_1)^j - (1 - p_2)^j}{p_2 - p_1} + (1-p_1)^j \\
    &= \frac{p_2 (1 - p_1)^j - p_1 (1 - p_2)^j}{p_2 - p_1}.
\end{align*}
\end{proof}

\begin{lemma}
For $x \ge 0$ and $j \ge 2$,
\begin{align}\label{ineq-1}
    1 - \frac{j-1}{x + j/2} \le \left(\frac{x}{x+1}\right)^{j-1}.
\end{align}
\end{lemma}

\begin{proof}
We just need to show that 
\begin{align*}
    (x + 1)^{j - 1}(x + j/2) - (j-1)(x+1)^{j-1} \le x^{j-1}(x + j/2),
\end{align*}
that is
\begin{align*}
    &(x+1)^j - j (x+1)^{j-1}/2 - x^{j-1}(x + j/2) \le 0, \\
    &\text{ i.e., } \sum_{l=0}^{j-2} \left[ \binom{j}{l} - \frac{j}{2} \binom{j-1}{l} \right] x^l \le 0.
\end{align*}
Note that for $l \le j-2$,
\begin{align*}
    \binom{j}{l} - \frac{j}{2} \binom{j-1}{l} = \left(1 - \frac{j-l}{2}\right)\binom{j}{l} \le 0,
\end{align*}
thus inequality (\ref{ineq-1}) hosts for $x \ge 0$ and $j \ge 2$.
\end{proof}

\begin{lemma}\label{lem:two-geo}
Let $X_1 \sim \geo(p_1), X_2 \sim \geo(p_2), Y_1, Y_2 \sim \geo\left(\frac{p_1 + p_2}{2}\right)$ be independent random variables 
with $0 < p_1 \le p_2 \le 1$. Then for any positive integer $j$, we have
\begin{align*}
    \sP\left( X_1 + X_2 > j \right) \ge \sP\left( Y_1 + Y_2 > j \right).
\end{align*}
\end{lemma}

\begin{proof}
If $j = 1$, 
then $\sP\left( X_1 + X_2 > j \right) = 1 = \sP\left( Y_1 + Y_2 > j \right)$. \\
If $p_1 = p_2 = 1$, then $\sP\left( X_1 + X_2 > j \right) = 0 = \sP\left( Y_1 + Y_2 > j \right)$ for $j \ge 2$. 

\vskip 10pt
Let $j \ge 2$, and $c \triangleq p_1 + p_2 < 2$ be a given constant. 

We prove that $f(p_1) \triangleq \pr{X_1 + X_2 > j}$ is a decreasing function.

\vskip 10pt
Employing equation (\ref{pr-neq}), for $p_1 < c/2$, we have
\begin{align*}
    f(p_1) = \frac{(c - p_1)(1 - p_1)^j - p_1(1 + p_1 - c)^j}{c - 2p_1},
\end{align*}
and 
\begin{align*}
    f'(p_1) &= \frac{-(1 - p_1)^j - j(c - p_1)(1 - p_1)^{j-1} 
    - (1 + p_1 - c)^j - jp_1(1 + p_1 - c)^{j-1}}{c - 2p_1} \\
    &+ 2\frac{(c - p_1)(1 - p_1)^j - p_1(1 + p_1 - c)^j}{(c - 2p_1)^2} \\
    &= \frac{[c(1 - p_1) - j (c - p_1)(c - 2p_1)](1 - p_1)^{j-1}
    - [c(1 + p_1 - c) + jp_1(c - 2p_1)](1 + p_1 - c)^{j-1}}{(c - 2p_1)^2}.
\end{align*}

Hence $f'(p_1) < 0$ is equivalent to
\begin{align}\label{lem:two-geo:eq}
    \frac{c(1 - p_1) - j (c - p_1)(c - 2p_1)}{c(1 + p_1 - c) + jp_1(c - 2p_1)} 
    < \left( \frac{1 + p_1 - c}{1 - p_1} \right)^{j-1}.
\end{align}

Note that
\begin{align*}
    &\quad \frac{c(1 - p_1) - j (c - p_1)(c - 2p_1)}{c(1 + p_1 - c) + jp_1(c - 2p_1)} \\
    &= 1 - \frac{(j-1)c(c - 2p_1)}{c(1 + p_1 - c) + jp_1(c - 2p_1)} \\
    &= 1 - \frac{j-1}{\frac{1 + p_1 - c}{c - 2p_1} + j\frac{p_1}{c}}
\end{align*}

Denote $x = \frac{1 + p_1 - c}{c - 2p_1}$. If $c \le 1$, then $p_1 > 0$ and $x > \frac{1-c}{c} \ge 0$. 
And if $c > 1$, then $p_1 \ge c - 1$ and $x \ge \frac{1 + c - 1 - c}{2 - c} = 0$. \\
Rewrite inequality (\ref{lem:two-geo:eq}) as 
\begin{align*}
    1 - \frac{j-1}{x + j p_1 / c} < \left(\frac{x}{x+1}\right)^{j-1}.
\end{align*}

Recall inequality (\ref{ineq-1}), we have
\begin{align*}
    \left(\frac{x}{x+1}\right)^{j-1} \ge 1 - \frac{j-1}{x + j/2} > 1 - \frac{j-1}{x + j p_1 / c}.
\end{align*}
Consequently, $f'(p_1) < 0$ hosts for $p_1 < c/2$ and $j \ge 2$. \\
With the fact that $\lim_{p_1 \rightarrow c/2} f(p_1) = f(c/2)$ according to equation (\ref{pr-eq}), 
we have 
\begin{align*}
    \sP\left( X_1 + X_2 > j \right) \ge \sP\left( Y_1 + Y_2 > j \right).
\end{align*}
for any positive integer $j$ and $0 < p_1 \le p_2 \le 1$.
\end{proof}

\begin{corollary}\label{coro:two-geo-plusz}
Let $X_1 \sim \geo(p_1), X_2 \sim \geo(p_2), Y_1, Y_2 \sim \geo\left(\frac{p_1 + p_2}{2}\right)$ be independent random variables with $0 < p_1 \le p_2 \le 1$.
Suppose $Z$ is a random variable that takes nonnegative integer values, 
and $Z$ is independent with $X_1, X_2, Y_1, Y_2$.
Then for any positive integer $j$, we have
\begin{align*}
    \pr{Z + X_1 + X_2 > j} \ge \pr{Z + Y_1 + Y_2 > j}.
\end{align*}
\end{corollary}

\begin{proof}
With applying Lemma \ref{lem:two-geo}, we have
\begin{align*}
    \pr{Z + X_1 + X_2 > j} &= \sum_{l=0}^{j - 1} \pr{Z = l} \pr{X_1 + X_2 > l - j} + \pr{Z > j - 1} \\
    &\ge \sum_{l=0}^{j - 1} \pr{Z = l} \pr{Y_1 + Y_2 > l - j} + \pr{Z > j - 1} \\
    &= \pr{Z + Y_1 + Y_2 > j}.
\end{align*}
\end{proof}

\begin{corollary}
Let $\{X_i\}_{1 \le i \le m}$ be independent variables, and $X_i$ follows a geometric distribution with success probability $p_i$. 
For any positive integer $j$, we have
\begin{align*}
    \pr{\sum_{i=1}^m X_i \ge j} \ge \pr{\sum_{i=1}^m Y_i \ge j},
\end{align*}
where $\{Y_i\}_{1 \le i \le m}$ are i.i.d. random variables, $Y_i \sim \geo(\sum_{i=1}^m p_i/m)$, 
and $Y_i$ is independent with $X_{i'} (1 \le i' \le m)$.
\end{corollary}

\begin{proof}
Let 
\begin{align*}
    f(p_1, p_2, \cdots, p_m) \triangleq \pr{\sum_{i=1}^m X_i \ge j}.
\end{align*}
Our goal is to minimize $f(p_1, p_2, \cdots, p_m)$ such that $\sum_{i=1}^m p_i = S \le m$.

By Corollary \ref{coro:two-geo-plusz}, we know that 
\begin{align*}
    f(p_1, p_2, \cdots, p_i, \cdots, p_j, \cdots, p_m) 
    \ge f(p_1, p_2, \cdots, \frac{p_i + p_j}{2}, \cdots, \frac{p_i + p_j}{2}, \cdots, p_m).
\end{align*}
This fact implies that $(p_1, p_2, \cdots, p_m)$ such that $p_1 = p_2 = \cdots = p_m = S/m$ is a minimizer of the function $f$.

\end{proof}

\begin{lemma}
Let $\{X_i\}_{1 \le i \le m}$ be i.i.d. random variables, 
and $X_i$ follows a geometric distribution with success probability $p$. 
We have
\begin{align}
    \sP\left(\sum_{i=1}^m X_i > \frac{m}{4p} \right) \ge 1 - \frac{16}{9m}
\end{align}
\end{lemma}

\begin{proof}
Denote $\sum_{i=1}^m X_i$ by $\tau$. We know that 
\begin{align*}
    \E \tau = \frac{m}{p}, ~~\Var (\tau) = \frac{m(1-p)}{p^2}.
\end{align*}
Hence, we have
\begin{align*}
    \sP\left(\tau > \frac{1}{4} \E \tau\right) &= \sP\left(\tau - \E \tau > -\frac{3}{4} \E \tau\right) \\
     &= 1 - \sP\left(\tau - \E \tau \le -\frac{3}{4} \E \tau\right) \\
     &\ge 1 - \sP\left(|\tau - \E \tau| \ge \frac{3}{4} \E \tau\right) \\
     &\ge 1 - \frac{16\Var(\tau)}{9(\E \tau)^2} \\
     &= 1 - \frac{16m(1-p)}{9m^2} \ge 1 - \frac{16}{9m}.
\end{align*}
\end{proof}

\begin{corollary}
Let $\{X_i\}_{1 \le i \le m}$ be independent random variables, 
and $X_i$ follows a geometric distribution with success probability $p_i$. 
Then
\begin{align*}
    \pr{\sum_{i=1}^m X_i > \frac{m^2}{4(\sum_{i=1}^m p_i)}} \ge 1 - \frac{16}{9m}.
\end{align*}
\end{corollary}

\section{Proof of Proposition \ref{prop:strongly}}\label{appendix:strongly}
\begin{proof}[\textbf{Proof of Proposition \ref{prop:strongly}}]
$\quad$
\begin{enumerate}
    \item Just recall Proposition \ref{prop:base}.
    \item Denote $\xi = \sqrt{\frac{2\Delta n(\alpha+1)^2}{(L-\mu)(\alpha-1)}}$. 

Let $\nabla F_{\text{SC}}(\vx) = 0$, that is
\begin{align*}
    \left( \frac{L-\mu}{2n}\A \left( \sqrt{\frac{2}{\alpha+1}} \right) + \mu \I \right)\vx = \frac{L-\mu}{n(\alpha+1)}\xi \ve_m, 
\end{align*}
or
\begin{align}\label{proof:strongly:minimizer0}
    \begin{bmatrix}
        \omega^2 + 1 + \frac{2n\mu}{L-\mu} & -1 & & & \\
        -1 & 2 + \frac{2n\mu}{L-\mu} & -1 & & \\
        & \ddots & \ddots & & \\
        & & -1 & 2 + \frac{2n\mu}{L-\mu} & -1 \\
        & & & -1 & 1 + \frac{2n\mu}{L-\mu}
    \end{bmatrix}
    \vx = 
    \begin{bmatrix}
        0 \\
        0 \\
        \vdots \\
        0 \\
        \frac{2\xi}{\alpha + 1}
    \end{bmatrix}
\end{align}

Note that $q = \frac{\alpha - 1}{\alpha + 1}$ is a root of the equation
\begin{align*}
    z^2 - \left(2 + \frac{2n\mu}{L-\mu}\right) z + 1 = 0,
\end{align*}
and 
\[\omega^2 + 1 + \frac{2n\mu}{L-\mu} = \frac{1}{q}, \]
\[\frac{2}{\alpha + 1} = 1 - q = -q^2 + (1 + \frac{2n\mu}{L-\mu})q.\]

Hence, it is easily to check that the solution to Equation (\ref{proof:strongly:minimizer0}) is 
$$\vx^{*} = \xi (q^{m}, q^{m-1}, \cdots, q)^{\top},$$
and 
$$ F_{\text{SC}}(\vx^*) = -\frac{L-\mu}{2n(\alpha+1)}\xi^2 q = -\Delta. $$

    \item 

If $\vx \in \gF_k$, $1 \le k < m$, then $x_1 = x_2 = \cdots = x_{m-k} = 0$. 

Let $\vy = \vx_{m-k+1:m} \in \sR^k$ and $\A_k$ be last $k$ rows and columns of the matrix in Equation (\ref{proof:strongly:minimizer}).
Then we can rewrite $F(\vx)$ as
\begin{align*}
    F_k(\vy) \triangleq F_{\text{SC}}(\vx) = \frac{L-\mu}{4n} \vy^{\top} \A_k \vy - \frac{L-\mu}{n(\alpha+1)}\xi \inner{\ve_m}{\vy}.
\end{align*}
Let $\nabla F_k (\vy) = 0$, that is
\begin{align}\label{proof:strongly:minimizer}
    \begin{bmatrix}
        2 + \frac{2n\mu}{L-\mu} & -1 & & & \\
        -1 & 2 + \frac{2n\mu}{L-\mu} & -1 & & \\
        & \ddots & \ddots & & \\
        & & -1 & 2 + \frac{2n\mu}{L-\mu} & -1 \\
        & & & -1 & 1 + \frac{2n\mu}{L-\mu}
    \end{bmatrix}
    \vy = 
    \begin{bmatrix}
        0 \\
        0 \\
        \vdots \\
        0 \\
        \frac{2\xi}{\alpha + 1}
    \end{bmatrix}.
\end{align}

By some calculation, the solution to above equation is 
\begin{align*}
    \frac{\xi q^{k+1}}{1+q^{2k+1}} \left( q^{-1} - q, q^{-2} - q^2, \cdots, q^{-k} - q^{k} \right)^{\top}.
\end{align*}

Thus 
\[\min_{\vx \in \gF_k} F_{\text{SC}}(\vx) = \min_{\vy \in \sR^k} F_k(\vy) = -\frac{L-\mu}{2n(\alpha+1)} \xi^2 q \frac{1 - q^{2k}}{1 + q^{2k+1}} = \Delta \frac{1 - q^{2k}}{1 + q^{2k+1}},\]
and 
\begin{align*}
    \min_{\vx \in \gF_k} F_{\text{SC}}(\vx) - F_{\text{SC}}(\vx^{*}) &= \Delta \left( 1 - \frac{1-q^{2k}}{1+q^{2k+1}} \right) \\
    &= \Delta q^{2k}\frac{1+q}{1+q^{2k+1}} \\
    &\ge \Delta q^{2k}.
\end{align*}

\end{enumerate}
\end{proof}

\section{Proof of Proposition \ref{prop:convex}}\label{appendix:convex}
\begin{proof}[\textbf{Proof of Proposition \ref{prop:convex}}]
$\quad$
\begin{enumerate}
    \item Just recall Proposition \ref{prop:base}.
    \item Denote $\xi = \frac{\sqrt{3}}{2}\frac{BL}{(m+1)^{3/2}n}$. Let $\nabla F_{\text{C}}(\vx) = 0$, that is
\begin{align*}
    \frac{L}{2n}\A(1) \vx = \frac{\xi}{n} \ve_m, 
\end{align*}
or
\begin{align}\label{proof:convex:minimizer}
    \begin{bmatrix}
        2 & -1 & & & \\
        -1 & 2 & -1 & & \\
        & \ddots & \ddots & & \\
        & & -1 & 2 & -1 \\
        & & & -1 & 1
    \end{bmatrix}
    \vx = 
    \begin{bmatrix}
        0 \\
        0 \\
        \vdots \\
        0 \\
        \frac{2\xi}{L}
    \end{bmatrix}.
\end{align}

Hence, it is easily to check that the solution to Equation (\ref{proof:convex:minimizer}) is 
$$\vx^{*} = \frac{2\xi}{L} (1, 2, \cdots, m)^{\top},$$
and 
$$ F_{\text{C}}(\vx^*) = -\frac{m\xi^2}{n L}. $$

Moreover, we have
\begin{align*}
    \norm{\vx_0 - \vx^*}^2 &= \frac{4\xi^2}{L^2} \frac{m(m+1)(2m+1)}{6} \\
    &\le \frac{4\xi^2}{3L^2}(m+1)^3 = B^2.
\end{align*}

    \item By similar calculation to above proof, we have
$$\argmin_{\vx \in \gF_k} F_{\text{C}}(\vx) = \frac{2\xi}{L} (1, 2, \cdots, k)^{\top},$$
and 
\begin{align*} 
\min_{\vx \in \gF_k} F_{\text{C}}(\vx) = - \frac{k \xi^2}{nL}.
\end{align*}

Thus 
\begin{align*}
    \min_{\vx \in \gF_k} F_{\text{C}}(\vx) - F_{\text{C}}(\vx^*) = \frac{\xi^2}{n L} (m - k).
\end{align*}

\end{enumerate}
\end{proof}

\section{Non-convex Case}\label{sec:nonconvex}
In non-convex case, our goal is to find an $\eps$-approximate stationary point $\hat{\vx}$ of our objective function $f$, which satisfies 
\begin{align}
    \norm{\nabla f(\hat{\vx})} \le \eps.
\end{align}

\subsection{Preliminaries}
We first introduce a general concept about smoothness.
\begin{defn}\label{defn:general_smooth}
For any differentiable function $f:\sR^{m+1} \rightarrow \sR$,
we say $f$ is $(l, L)$-smooth, if for any $\vx, \vy \in \sR^m$ we have
\[\frac{l}{2}\norm{\vx - \vy}^2 \le f(\vx) - f(\vy) - \inner{\nabla f(\vy)}{\vx - \vy} \le \frac{L}{2}\norm{\vx - \vy}^2,\]
where $L > 0, l \in \sR$.
\end{defn}
Especially, if $f$ is $L$-smooth, then it can be checked that $f$ is $(-L, L)$-smooth.

If $f$ is $(-\sigma, L)$-smooth, in order to make the operator $\prox_{f}^{\gamma}$ valid, we set $\frac{1}{\gamma} > \sigma$ to ensure the function 
\begin{align*}
    \hat{f}(\vu) \triangleq f(\vu) + \frac{1}{2\gamma} \norm{\vx - \vu}^2
\end{align*}
is a convex function.

Next, we introduce a class of function which is original proposed in \citep{carmon2017lower}. Let $G_{\text{NC}}: \sR^{m+1} \rightarrow \sR$ be
\begin{align*}
    G_{\text{NC}}(\vx; \alpha, m) = \frac{1}{2} \norm{\B(m+1, \sqrt[4]{\alpha}) \vx}^2 - \sqrt{\alpha} \inner{\ve_1}{\vx} + \alpha \sum_{i=1}^{m} \Gamma (x_i),
\end{align*}
where the non-convex function $\Gamma: \sR \rightarrow \sR$ is
\begin{align}
    \Gamma(x) \triangleq 120 \int_1^{x} \frac{t^2(t-1)}{1 + t^2} d t.
\end{align}

We need following properties about $G_{\text{NC}}(\vx; \alpha, m)$.
\begin{prop}[Lemmas 3,4, \cite{carmon2017lower}]\label{prop:nonconvex:prop:base}
For any $0 < \alpha \le 1$, it holds that
\begin{enumerate}
    \item $\Gamma(x)$ is $(-45 (\sqrt{3} - 1), 180)$-smooth and $G_{\text{NC}}(\vx; \alpha, m)$ is $(-45 (\sqrt{3} - 1) \alpha , 4 + 180 \alpha)$-smooth.
    \item $G_{\text{NC}}(\vzero; \alpha, m) - \min_{\vx \in \sR^{m+1}} G_{\text{NC}}(\vx; \alpha, m) \le \sqrt{\alpha} / 2 + 10 \alpha m$.
    \item For $\vx$ which satisfies that $x_m = x_{m+1} = 0$, we have 
    \[\norm{\nabla G_{\text{NC}}(\vx; \alpha, m)} \ge \alpha^{3/4}/4.\]
\end{enumerate}
\end{prop}

\subsection{Our Result}
\begin{thm}\label{thm:nonconvex}
For any PIFO algorithm $\gA$ and any $L, \sigma, n, \Delta, \eps$ such that
$\eps^2 \le \frac{\Delta L \alpha}{81648 n}$,
there exist a dimension $d = \floor{\frac{\Delta L \sqrt{\alpha}}{40824 n \eps^2}} + 1$ and $n$ $(-\sigma, L)$-smooth nonconvex functions $\{f_i:\sR^d\rightarrow\sR\}_{i=1}^n$  such that $f(\vx_0) - f(\vx^*) \le \Delta$. In order to find $\hat{\vx} \in \sR^d$ such that $\E \norm{\nabla f(\hat{\vx})} < \eps$, $\gA$ needs at least $\Omega \left(\frac{\Delta L \sqrt{\alpha}}{\eps^2}\right)$ queries to $h_f$, where we set $\alpha = \min \left\{1, \frac{(\sqrt{3} + 1)n\sigma}{30 L}, \frac{n}{180} \right\}$.
\end{thm}

\begin{remark}
For $n > 180$, wehave
\begin{align*}
    \Omega \left(\frac{\Delta L \sqrt{\alpha}}{\eps^2}\right) = \Omega \left( \frac{\Delta}{\eps^2} \min \left\{ L, \sqrt{\frac{\sqrt{3} + 1}{30}} \sqrt{n \sigma L}, \frac{\sqrt{n}L}{\sqrt{180}} \right\} \right) = \Omega \left( \frac{\Delta}{\eps^2} \min\{L, \sqrt{n\sigma L}\} \right).
\end{align*}
Thus, our result is comparable to the one of \citeauthor{zhou2019lower}'s result (their result only related to IFO algorithms, so our result is more strong), but our construction only requires the dimension be $\gO \left(1 + \frac{\Delta}{\eps^2} \min \{L/n, \sqrt{\sigma L /n}\} \right)$, which is much smaller than 
$\gO \left( \frac{\Delta}{\eps^2} \min \{L, \sqrt{n \sigma L}\} \right) $ in \citep{zhou2019lower}.
\end{remark}

\subsection{Constructions}
Consider
\begin{align}
    F(\vx; \alpha, m, \lambda, \beta) = \lambda G_{\text{NC}} (\vx / \beta; \alpha, m).    
\end{align}

Similar to our construction we introduced in Section \ref{sec:pre}, we denote the $l$-th row of the matrix $\B(m+1, \sqrt[4]{\alpha})$ by $\vb_l$ and 
\begin{align}
    \gL_i = \{ l: 1 \le l \le m, m + 1 - l \equiv i(\bmod ~n) \}, i = 1,2, \cdots, n.
\end{align}

Let $\gG_k = \spn \{\ve_1, \ve_2, \cdots, \ve_k\}$, $1 \le k \le m$, $\gG_0 = \{\vzero\}$ and compose $F(\vx; \alpha, m, \lambda, \beta)$ to
\begin{align}
\begin{cases}
\!f_1(\vx; \alpha, m, \lambda, \beta) = \frac{\lambda n}{2\beta^2} \sum\limits_{l \in \gL_i} \norm{\vb_{l}^{\top}\vx}^2 
- \frac{\lambda n \sqrt{\alpha}}{\beta} \inner{\ve_1}{\vx} + \lambda \alpha \sum\limits_{i=1}^{m} \Gamma (x_i / \beta), \\
\!f_i(\vx; \alpha, m, \lambda, \beta) = \frac{\lambda n}{2\beta^2} \sum\limits_{l \in \gL_i} \norm{\vb_{l}^{\top} \vx}^2 + \lambda \alpha \sum\limits_{i=1}^{m} \Gamma (x_i/\beta), \text{ for } i \ge 2.
\end{cases}
\end{align}

Clearly, $F(\vx; \alpha, m, \lambda, \beta) = \frac{1}{n} \sum_{i=1}^n f_i(\vx; \alpha, m, \lambda, \beta)$. 
Moreover, by Proposition \ref{prop:nonconvex:prop:base}, we have following properties about $F(\vx; \alpha, m, \lambda, \beta)$ and $\{f_i(\vx; \alpha, m, \lambda, \beta)\}_{i=1}^n$.
\begin{prop}
For any $0 < \alpha \le 1$, it holds that
\begin{enumerate}
    \item $f_i(\vx; \alpha, m, \lambda, \beta)$ is $\left(\frac{-45 (\sqrt{3} - 1) \alpha\lambda}{\beta^2} , \frac{(2n + 180 \alpha)\lambda}{\beta^2}\right)$-smooth.
    \item $F(\vzero; \alpha, m, \lambda, \beta) - \min_{\vx \in \sR^{m+1}} F(\vx; \alpha, m, \lambda, \beta) \le \lambda(\sqrt{\alpha} / 2 + 10 \alpha m)$.
    \item For $\vx$ which satisfies that $x_m = x_{m+1} = 0$, we have 
    \[\norm{\nabla F(\vx; \alpha, m, \lambda, \beta)} \ge \frac{\alpha^{3/4} \lambda}{4 \beta}.\]
\end{enumerate}
\end{prop}

Similar to Lemma \ref{lem:jump}, similar conclusion hosts for $\{f_i(\vx; \alpha, m, \lambda, \beta)\}_{i=1}^n$.
\begin{lemma}\label{lem:nonconvex:jump}
For $\vx \in \gF_k$, $0 \le k < m$ and $\gamma < \frac{\sqrt{2} + 1}{60} \frac{\beta^2}{\lambda \alpha}$, 
we have
\begin{align*}
    \nabla f_{i}(\vx; \alpha, m, \lambda, \beta),  \prox_{f_{i}}^{\gamma} (\vx)\in 
    \begin{cases}
    \gG_{k+1}, \text{ if } k \equiv i - 1 (\bmod ~n), \\
    \gG_{k}, \text{ otherwise}.
    \end{cases}
\end{align*}
\end{lemma}

\begin{proof}
Let $G(\vx) \triangleq \sum\limits_{i=1}^{m} \Gamma (x_i)$ and $\Gamma'(x)$ be the derivative of $\Gamma(x)$.

First note that $\Gamma'(0) = 0$, so if $\vx \in \gG_k$, then 
\[\nabla G(\vx) = \big( \Gamma'(x_1), \Gamma'(x_2), \cdots, \Gamma'(x_m) \big)^{\top} \in \gG_k. \]

Moreover, for $\vx \in \gF_G ~(k \ge 1)$, we have 
\begin{align*}
    \vb_{l}^{\top} \vx &= 0 \text{ for } l < m - k, \\
    \vb_{l} &\in \gG_k \text{ for } l > m - k, \\ 
    \vb_{m - k} &\in \gG_{k+1}.
\end{align*}
Consequently, for $l \neq m - k$, $\vb_l \vb_l^{\top} \vx = (\vb_l^{\top} \vx) \vb_l \in \gG_k $,
and $\vb_{m - k} \vb_{m - k}^{\top} \vx \in \gG_{k+1} $.

For $k = 0$, we have $\vx = \vzero$, and
\begin{align*}
\nabla f_1(\vx) = \lambda n \sqrt{\alpha}/\beta ~\ve_1 \in \gG_1, \\
\nabla f_j(\vx) = \vzero ~(j \ge 2).
\end{align*}

For $k \ge 1$, we suppose that $m - k \in \gL_{i}$. Since
\begin{align*}
    \nabla f_j(\vx) = \frac{\lambda n}{\beta^2} \sum_{l \in \gL_j} \vb_l^{\top} \vb_l \vx + \frac{\lambda \alpha}{\beta} ~\nabla G(\vx/\beta) - \eta_j \ve_1,
\end{align*}
where $\eta_1 = \lambda n \sqrt{\alpha}/\beta$, $\eta_j = 0$ for $j \ge 2$. \\
Hence, $\nabla f_{i} (\vx) \in \gF_{k+1}$ and $\nabla f_j (\vx) \in \gF_{k} ~(j \neq i)$.

\vskip 5pt
Now, we turn to consider $\vv = \prox_{f_j}^{\gamma} (\vx)$. \\

We have
\begin{align*}
    \nabla f_j(\vv) + \frac{1}{\gamma} (\vv - \vx) = \vzero,
\end{align*}
that is
\begin{align}
    \left(\frac{\lambda n}{\beta^2} \sum_{l \in \gL_j} \vb_l^{\top} \vb_l + \frac{1}{\gamma} \I \right) \vv + \frac{\lambda\alpha}{\beta} \nabla G(\vv/\beta) = \eta_j \ve_1 + \frac{1}{\gamma} \vx.
\end{align}

Denote
\begin{align*}
\A = \frac{\lambda n}{\beta} \sum_{l \in \gL_j} \vb_l^{\top} \vb_l + \frac{\beta}{\gamma} \I, ~\vu = \frac{1}{\beta} \vv,  ~\vy = \eta_j \ve_1 + \frac{1}{\gamma} \vx, 
\end{align*}
then we have 
\begin{align}\label{eq:nonconvex:prox}
    \A \vu + \frac{\lambda \alpha}{\beta} \nabla G(\vu) = \vy.
\end{align}

Next, if $s$ satisfies
\begin{align}\label{s:condition}
\begin{cases}
s > \max \{1, k\} &\text{ for } j = 1, \\
s > k &\text{ for } j > 1,
\end{cases}
\end{align}
then we know that the $s$-th element of $\vy$ is $0$.

\vskip 10pt
If $s$ satisfies (\ref{s:condition}) and $m - s \in \gL_j$, then the $s$-th and $(s + 1)$-th elements of $\A \vu$ is 
$\left((\xi + \beta / \gamma) u_s - \xi u_{s+1}\right)$ and $\left(-\xi u_s + (\xi + \beta/\gamma) u_{s+1}\right)$ respectively where $\xi = \lambda n / \beta$.
So by Equation (\ref{eq:nonconvex:prox}), we have 
\begin{align*}
\begin{cases}
\frac{\beta}{\gamma} u_s + \xi (u_s - u_{s+1}) + \frac{120\lambda \alpha}{\beta} \frac{u_s^2(u_s-1)}{1+ u_s^2} = 0.  \\
\frac{\beta}{\gamma} u_{s+1} + \xi (u_{s+1} - u_s) + \frac{120\lambda \alpha}{\beta} \frac{u_{s+1}^2(u_{s+1}-1)}{1+ u_{s+1}^2} = 0.
\end{cases}
\end{align*}
Following from Lemma \ref{lem:solution:z1z2}, for $\frac{120\lambda \alpha}{\beta} < \frac{(2 + 2\sqrt{2})\beta}{\gamma}$, we have $u_s = u_{s+1} = 0$. \\
That is 
\begin{enumerate}
    \item if $m - s \in \gL_j$ and $s$ satisfies (\ref{s:condition}), then $u_s = 0$. 
    \item if $m - s + 1 \in \gL_j$ and $s - 1$ satisfies (\ref{s:condition}), then $u_s = 0$. 
\end{enumerate}
    
\vskip 10pt
For $s$ which satisfies (\ref{s:condition}), if $m - s \not\in \gL_j$ and $m - s + 1 \not\in \gL_j$, then the $s$-th element of $\A\vu$ is $(\beta/\gamma ~u_s)$.
Similarly, by Equation (\ref{eq:nonconvex:prox}), we have
\begin{align*}
    \frac{\beta}{\gamma} u_s + \frac{120\lambda \alpha}{\beta} \frac{u_s^2(u_s-1)}{1+ u_s^2} = 0.
\end{align*}
Following from Lemma \ref{lem:solution:z}, for $\frac{120\lambda \alpha}{\beta} < \frac{(2 + 2\sqrt{2})\beta}{\gamma}$, we have $u_s = 0$.

\vskip 10pt
Therefore, we can conclude that
\begin{enumerate}
    \item if $s - 1$ satisfies (\ref{s:condition}), then $u_s = 0$.
    \item if $s$ satisfies (\ref{s:condition}) and $m -s + 1 \not\in \gL_j$, then $u_s = 0$. 
\end{enumerate}

Moreover, we have that
\begin{enumerate}
    \item if $k = 0$ and $j = 1$, then $m - 1, m - 2 \not\in \gL_j$, so $u_2 = 0$.
    \item if $k = 0$ and $j > 1$, then for $s = 1$, we have $m - s + 1 \not\in \gL_j$, so $u_1 = 0$.
    \item if $k = 0$, then for $s > 2$, we have $s - 1 > 1$ satisfies (\ref{s:condition}), so $u_s = 0$.
    \item if $k > 0$, then for $s > k + 1$, we have $s - 1 > k$ satisfies (\ref{s:condition}), so $u_{s} = 0$. 
    \item if $k > 0$ and $m - k \not\in \gL_j$, then for $s = k + 1$, we have $m - s + 1 \not\in \gL_j$, so $u_{k+1} = 0$.
\end{enumerate}

In short, 
\begin{enumerate}
    \item if $k = 0$ and $j > 1$, then $\vu \in \gG_0$.
    \item if $k = 0$ and $j = 1$, then $\vu \in \gG_1$.
    \item if $k > 1$ and $m - k \not\in \gL_j$, then $\vu \in \gG_{k}$.
    \item if $k > 1$ and $m - k \in \gL_j$, then $\vu \in \gG_{k+1}$.
\end{enumerate}

\end{proof}

\begin{remark}
In order to make the operator $\prox_{f_i}^{\gamma}$ valid, $\gamma$ need to satisfy 
\[\gamma < \frac{\sqrt{3} + 1}{90} \frac{\beta^2}{\lambda \alpha} < \frac{\sqrt{2} + 1}{60} \frac{\beta^2}{\lambda \alpha}. \]
So for any valid PIFO call, the condition about $\gamma$ in Lemma \ref{lem:nonconvex:jump} must be satisfied.
\end{remark}

\begin{lemma}\label{lem:solution:z}
Suppose that $0 < \lambda_2 < (2 + 2\sqrt{2}) \lambda_1$, 
then $z = 0$ is the only real solution to the equation 
\begin{align}\label{eq:solution:z}
\lambda_1 z + \lambda_2 \frac{z^2(z-1)}{1+ z^2} = 0.    
\end{align}
\end{lemma}
\begin{proof}
Since $0 < \lambda_2 < (2 + 2\sqrt{2}) \lambda_1$, we have 
\[\lambda_2^2 - 4 \lambda_1 (\lambda_1 + \lambda_2) < 0,\]
and consequently, for any $z$, $(\lambda_1 + \lambda_2)z^2 - \lambda_2 z + \lambda_1 > 0$.

On the other hand, we can rewrite Equation (\ref{eq:solution:z}) as 
\begin{align*}
    z \big((\lambda_1 + \lambda_2)z^2 - \lambda_2 z + \lambda_1\big) = 0.
\end{align*}
Clearly, $z = 0$ is the only real solution to Equation (\ref{eq:solution:z}).

\end{proof}

\begin{lemma}\label{lem:solution:z1z2}
Suppose that $0 < \lambda_2 < (2 + 2\sqrt{2}) \lambda_1$ and $\lambda_3 > 0$, 
then $z_1 = z_2 = 0$ is the only real solution to the equation 
\begin{align}\label{eq:solution:z1z2}
\begin{cases}
\lambda_1 z_1 + \lambda_3 (z_1 - z_2) + \lambda_2 \frac{z_1^2(z_1-1)}{1+ z_1^2} = 0.  \\
\lambda_1 z_2 + \lambda_3 (z_2 - z_1) + \lambda_2 \frac{z_2^2(z_2-1)}{1+ z_2^2} = 0.
\end{cases}
\end{align}
\end{lemma}

\begin{proof}
If $z_1 = 0$, then $z_2 = 0$. So let assume that $z_1 z_2 \neq 0$.
Rewrite the first equation of Equation (\ref{eq:solution:z1z2}) as
\begin{align*}
    \frac{\lambda_1 + \lambda_3}{\lambda_3} + \frac{\lambda_2}{\lambda_3} \frac{z_1(z_1-1)}{1+ z_1^2} = \frac{z_2}{z_1}
\end{align*}
Note that 
\begin{align*}
    \frac{1 - \sqrt{2}}{2} \le \frac{z(z-1)}{1+z^2}.
\end{align*}
Thus, we have
\begin{align*}
    \frac{\lambda_1 + \lambda_3}{\lambda_3} + \frac{\lambda_2}{\lambda_3} \frac{1 - \sqrt{2}}{2} 
    \le \frac{z_2}{z_1}.
\end{align*}
Similarly, it also holds
\begin{align*}
    \frac{\lambda_1 + \lambda_3}{\lambda_3} + \frac{\lambda_2}{\lambda_3} \frac{1 - \sqrt{2}}{2} 
    \le \frac{z_1}{z_2}.
\end{align*}

By $0 < \lambda_2 < (2 + 2\sqrt{2}) \lambda_1$, we know that $\lambda_1 + \frac{1 - \sqrt{2}}{2}\lambda_2 > 0$. 
Thus
\begin{align*}
    \frac{\lambda_1 + \lambda_3}{\lambda_3} + \frac{\lambda_2}{\lambda_3} \frac{1 - \sqrt{2}}{2} > 1.
\end{align*}
Since $z_1/z_2 > 1$ and $z_2 / z_1 > 1$ can not hold at the same time, so we get a contradiction.
\end{proof}

Following from Lemma \ref{lem:nonconvex:jump}, we know following Lemma which is similar to Lemma \ref{lem:base}.
\begin{lemma}\label{lem:nonconvex:base}
If $M \ge 1$ satisfies $\min_{\vx \in \gG_M} \norm{\nabla F(\vx)} \ge 9\eps$ and $N = n(M+1)/4$, 
then we have
\begin{align*}
    \min_{t \le N} \E \norm{\nabla F(\vx_t)} \ge \eps.
\end{align*}
\end{lemma}

\begin{thm}
Set 
\begin{align*}
    \alpha &= \min \left\{1, \frac{(\sqrt{3} + 1)n\sigma}{30 L}, \frac{n}{180} \right\}, \\
    \lambda &= \frac{3888 n \eps^2}{L \alpha^{3/2}}, \\
    \beta &= \sqrt{3 \lambda n / L}, \\
    m &= \floor{\frac{\Delta L \sqrt{\alpha}}{40824 n \eps^2}}
\end{align*}
Suppose that $\eps^2 \le \frac{\Delta L \alpha}{81648 n}$. In order to find $\hat{\vx} \in \sR^{m+1}$ such that $\E \norm{\nabla F(\hat{\vx})} < \eps$, PIFO algorithm $\gA$ needs at least $\Omega \left(\frac{\Delta L \sqrt{\alpha}}{\eps^2}\right)$ queries to $h_{F}$.
\end{thm}

\begin{proof}
First, note that $f_i$ is $(-l_1, l_2)$-smooth, where
\begin{align*}
    l_1 &= \frac{45 (\sqrt{3} - 1) \alpha\lambda}{\beta^2} = \frac{45 (\sqrt{3} - 1)L}{3n} \alpha 
    \le \frac{45 (\sqrt{3} - 1) L}{3n} \frac{(\sqrt{3} + 1)n\sigma}{30 L} = \sigma, \\
    l_2 &= \frac{(2n + 180 \alpha)\lambda}{\beta^2} = \frac{L}{3n} (2n + 180 \alpha) \le L.
\end{align*}
Thus each $f_i$ is $(-\sigma, L)$-smooth.

Next, observe that 
\begin{align*}
    F(\vx_0) - F(\vx^*) &\le \lambda (\sqrt{\alpha}/2 + 10 \alpha m) = \frac{1944 n \eps^2}{L \alpha} + \frac{38880 n \eps^2}{L \sqrt{\alpha}} m \\
    &\le \frac{1944}{40824} \Delta + \frac{38880}{40824} \Delta = \Delta.
\end{align*}

For $M = m - 1$, we know that 
\begin{align*}
    \min_{\vx \in \gG_M} \norm{\nabla F(\vx)} \ge \frac{\alpha^{3/4}\lambda}{4\beta} 
    = \frac{\alpha^{3/4}\lambda}{4\sqrt{3\lambda n/L}} = \sqrt{\frac{\lambda L}{3n}} \frac{\alpha^{3/4}}{4} = 9 \eps.
\end{align*}

With recalling Lemma \ref{lem:nonconvex:base}, in order to find $\hat{\vx} \in \sR^{m+1}$ such that $\E \norm{\nabla F(\hat{\vx})} < \eps$, PIFO algorithm $\gA$ needs at least $N$ queries to $h_{F}$, where
\begin{align*}
    N &= n(M + 1)/4 = n m /4 = \Omega \left(\frac{\Delta L \sqrt{\alpha}}{\eps^2}\right).
\end{align*}

At last, we need to ensure that $m \ge 2$. By $\eps^2 \le \frac{\Delta L \alpha}{81648 n}$, we have
\begin{align*}
    \frac{\Delta L \sqrt{\alpha}}{40824 n \eps^2} \ge \frac{\Delta L \alpha}{40824 n \eps^2} \ge 2,
\end{align*}
and consequently $m \ge 2$.
\end{proof}

\end{document}